\documentclass[11pt]{article}

\usepackage[margin=1.1in]{geometry}
\usepackage{amsmath,amssymb,amsfonts,amsthm,mathtools}
\usepackage{bm}
\usepackage{hyperref}
\usepackage{enumitem}
\usepackage{mathrsfs}
\usepackage{microtype}

\hypersetup{
  colorlinks=true,
  linkcolor=blue,
  citecolor=blue,
  urlcolor=blue
}

\newtheorem{theorem}{Theorem}[section]
\newtheorem{proposition}[theorem]{Proposition}
\newtheorem{lemma}[theorem]{Lemma}
\newtheorem{corollary}[theorem]{Corollary}
\newtheorem{definition}[theorem]{Definition}
\newtheorem{assumption}[theorem]{Assumption}
\newtheorem{remark}[theorem]{Remark}

\newcommand{\R}{\mathbb{R}}

\newcommand{\E}{\mathbb{E}}
\newcommand{\Tr}{\operatorname{Tr}}
\newcommand{\FP}{\operatorname{FP}}
\newcommand{\Dom}{\operatorname{Dom}}
\newcommand{\Id}{\operatorname{Id}}
\newcommand{\dd}{\,\mathrm{d}}
\newcommand{\braket}[2]{\left\langle #1,#2\right\rangle}
\newcommand{\norm}[1]{\left\lVert #1\right\rVert}

\title{A Quadratic-Form Representation of the Scalar Casimir Trace
from Codimension-Three Riesz Reduction}
\author{
Irshadullah Khan\thanks{\url{https://www.researchgate.net/profile/Irshadullah_Khan/research}}\\
Department of Mathematics\\
Quaid-i-Azam University\thanks{Visiting Faculty under the UNDP TOKTEN programme.}\\
Islamabad, Pakistan\\
\texttt{irshadk2@gmail.com}
\and
\href{https://orcid.org/0000-0001-8823-6752}{Bilal Khan}\thanks{\url{https://engineering.lehigh.edu/faculty/bilal-khan}}\\
Department of Computer Science\\
Lehigh University\\
Bethlehem, PA, USA\\
\texttt{bik221@lehigh.edu}
}

\date{}

\begin{document}
\maketitle

\begin{abstract}
Under a prescribed heat-regularized Gaussian source covariance, we give a
quadratic-form representation of the scalar Casimir trace associated with a
codimension-three Riesz reduction.  For a product operator
\(L_M=L_B-\Delta_\perp\), with \(L_B\) positive self-adjoint and bounded
below, transverse reduction of the ambient Riesz operator \(L_M^{-s}\)
produces the brane multiplier \(L_B^{m/2-s}\), up to an explicit
Gamma-function constant.  The exponent \(s=1+m/2\) is therefore the critical
Riesz exponent for obtaining the ordinary brane Green operator \(L_B^{-1}\);
in codimension three this gives \(s=5/2\).

Using this induced Green kernel, we prescribe a Gaussian generalized scalar
source with covariance proportional to \(L_B^{3/2}e^{-\tau L_B}\).  The
expectation of its quadratic Green-kernel energy is then exactly the
heat-regularized scalar Casimir trace
\[
    \frac{\hbar c}{2}
    \operatorname{Tr}\!\left(L_B^{1/2}e^{-\tau L_B}\right).
\]
With the same finite-part prescription, the identity specializes in the
Dirichlet parallel-plate geometry to the standard scalar finite part.

We also record a deterministic flat Green-energy calibration at the plate
scale.  Within the plate-compatible rectangular aspect-ratio family, the
cubical cell is selected by spectral, heat-trace, and Green-energy extremal
criteria, and the associated comparison coefficient is the corresponding
extremal calibration value.  The construction is a scalar spectral
representation theorem; no electromagnetic, gravitational, brane-dynamical,
or fundamental-constant identification is asserted.
\end{abstract}

\section{Introduction}

This note studies a scalar spectral representation of Casimir-type trace
functionals.  The basic observation is that, for a product operator
\(L_M=L_B-\Delta_\perp\) on \(B\times\R^m_\perp\), the transverse momentum
restriction of the fractional operator \(L_M^{-s}\) is a spectral multiplier
of \(L_B\).  In codimension three, the choice \(s=5/2\) gives an induced
brane operator proportional to \(L_B^{-1}\).

The second observation is stochastic.  If a heat-regularized Gaussian
generalized scalar source is assigned covariance proportional to
\(L_B^{3/2}e^{-\tau L_B}\), then the expectation of its quadratic energy with
respect to the induced Green operator \(gL_B^{-1}\) is exactly
\((\hbar c/2)\operatorname{Tr}(L_B^{1/2}e^{-\tau L_B})\).  Thus the construction
represents the regulated scalar Casimir trace as an expected quadratic form.

The individual ingredients used below are standard: heat-kernel and
zeta-function finite parts in Casimir theory, spectral calculus for positive
self-adjoint operators, Riesz-type kernel integrals, and Gaussian generalized
fields.  The point of the present note is not to introduce a new
regularization method or a new physical plate model, but to assemble these
ingredients into a single scalar representation theorem in which the
regularized trace is realized as the expectation of a quadratic form.

The parallel-plate specialization is included as a scalar benchmark.  It
fixes the standard Dirichlet scalar finite part used later when expressing the
finite-part trace in deterministic flat Green-energy units.  The corresponding
reference functional uses the same inverse-distance Green-kernel scaling as
the flat brane operator \(L_0^{-1}\); the cube appears as the unit reference
cell compatible with the normalization \(A=n^2a^2\).

The final part of the paper therefore has a narrower calibration purpose.
It does not enter the proof of the stochastic trace identity.  Rather, it
asks how the scalar finite-part Green energy obtained from the trace
representation compares with a deterministic flat \(L_0^{-1}\) Green energy
defined at the same plate scale.  Within the plate-compatible rectangular
aspect-ratio family, the cubical reference cell is characterized by spectral,
heat-trace, and Green-energy extremal properties.  The resulting comparison
coefficient is the extremal value of this calibration functional on that
restricted family.

The contribution is therefore organizational rather than a new
regularization method or a new physical plate model.  The paper isolates a
codimension-three Riesz reduction that produces the brane Green operator,
uses a prescribed heat-regularized Gaussian source covariance to realize the
scalar Casimir trace as an expected quadratic Green energy, and then
calibrates the resulting finite part against a deterministic flat Green
energy at the same plate scale.  The stochastic identity is useful in this
paper as a representation principle: it rewrites the scalar spectral trace
as a quadratic-form expectation with all constants and finite-part
normalizations explicit.

The result should be read as a representation theorem for scalar spectral
functionals.  The fractional ambient operator is not an ordinary local
six-dimensional propagator, and the Gaussian source covariance is a specified
convention rather than a derivation from a microscopic quantum field model.
No electromagnetic, gravitational, brane-dynamical, or fundamental-constant
identification is claimed.

\section{Standing framework and conventions}
\label{sec:framework}

The purpose of the present note is to formulate and prove a family of scalar operator identities and stochastic trace identities.  The framework consists of a brane Hilbert space, an ambient product operator, restricted Riesz-type mediators, and heat-regularized Gaussian generalized scalar sources.  The terms ``brane'', ``ambient'', ``mediator'', and ``source'' are used as compact mathematical labels for these objects.  Throughout the paper, the word ``brane'' denotes only the distinguished Hilbert-space factor \(B\) in the product geometry \(B\times\R^3_\perp\).  No brane dynamics, embedding equations, gravitational backreaction, or string-theoretic brane model is assumed.  Inverse-distance kernels appearing later are used as geometric normalization functionals.  Equalities involving scalar finite parts are equalities of the corresponding regularized and renormalized spectral quantities under the prescriptions stated below.

\subsection{Brane operator}

Let $B$ be a three-dimensional Riemannian manifold or a compactified three-dimensional region.  Let
\[
    H_B := L^2(B,\dd\mu_B)
\]
be the brane Hilbert space.  We assume that $L_B$ is a positive self-adjoint operator on $H_B$.  In the motivating examples one may take
\[
    L_B=-\Delta_B
\]
with specified boundary conditions.  The hypotheses used in the sequel are the following.

\begin{assumption}[Spectral hypotheses on $L_B$]
\label{ass:L_B}
The operator $L_B$ is self-adjoint and strictly positive: there is a constant
\[
    \ell_B>0
    \qquad\text{such that}\qquad
    L_B\geq \ell_B\Id
\]
in the sense of quadratic forms.  In particular, $L_B^{-1}$ is a bounded positive operator on $H_B$.  In Sections~\ref{sec:stochastic}--\ref{sec:renormalized}, for notational simplicity, we first treat the case of compact resolvent,
\[
    L_B u_j=\lambda_j u_j, \qquad
    \ell_B\leq\lambda_1\leq \lambda_2\leq\cdots,
\]
where $\{u_j\}_{j\geq1}$ is an orthonormal basis of $H_B$.  We further assume that
\[
    \Tr(e^{-\tau L_B})<\infty \qquad (\tau>0).
\]
The non-compact parallel-plate limit is obtained in Section~\ref{sec:parallel-plates} by imposing lateral periodic boundary conditions on a finite box, retaining the Dirichlet spectral gap in the normal direction, and then passing to energy per unit area.
\end{assumption}

\begin{remark}[Zero modes and reduced inverses]
If a non-negative operator $L_B$ has a finite-dimensional kernel, the same formal identities may be written on $\ker(L_B)^{\perp}$ with $L_B^{-1}$ replaced by the reduced inverse.  If there is no positive lower spectral bound, however, $L_B^{-1}$ is generally unbounded and the quadratic forms require additional domain hypotheses.  The main text avoids this complication by imposing the spectral gap in Assumption~\ref{ass:L_B}.
\end{remark}

\subsection{Ambient product geometry}

Let the ambient space be
\[
    M := B\times \R^3_{\perp},
\]
where $\R^3_\perp$ denotes three transverse directions.  The ambient Hilbert space is
\[
    H_M := H_B\otimes L^2(\R^3_\perp).
\]
Define the ambient product operator
\begin{equation}
    L_M := L_B\otimes \Id + \Id\otimes (-\Delta_\perp),
    \label{eq:ambient-operator}
\end{equation}
where $-\Delta_\perp$ is the non-negative Euclidean Laplacian on the transverse factor.  The fractional powers of $L_M$ are understood by the spectral theorem.

\begin{remark}[Why product geometry is assumed]
The product form \eqref{eq:ambient-operator} is a structural assumption, not a conclusion.  It is chosen because it permits a transparent computation of the brane-restricted ambient operator.  The construction is a formal representation theorem under this assumption.
\end{remark}

\subsection{Ambient mediator and brane restriction}

For $s>3/2$ define the ambient Riesz-type mediator
\[
    V_{M,s}:=\kappa_s L_M^{-s},
\]
where $\kappa_s$ is an overall normalization constant.  The word ``mediator'' is used for the spectral kernel $L_M^{-s}$ and for its brane-restricted operator.  It is not meant to imply an ordinary local propagator unless explicitly stated.  After Fourier transform in the transverse variables, $L_M^{-s}$ is the direct integral of the bounded operators $(L_B+|q|^2)^{-s}$.  The brane-restricted mediator is defined directly by the operator-valued integral
\begin{equation}
    V_{B,s}
    :=\kappa_s\int_{\R^3}\frac{\dd^3q}{(2\pi)^3}
      (L_B+|q|^2)^{-s}.
    \label{eq:brane-restriction-operator-integral}
\end{equation}
This integral is an operator-norm Bochner integral.  Indeed, by Assumption~\ref{ass:L_B},
\[
    \bigl\|(L_B+|q|^2)^{-s}\bigr\|
    \leq (\ell_B+|q|^2)^{-s},
\]
and the right-hand side is integrable over $\R^3$ exactly when $s>3/2$.  Equivalently, for $f,h\in H_B$,
\begin{equation}
    \braket{f}{V_{B,s}h}
    =\kappa_s\int_{\R^3}\frac{\dd^3q}{(2\pi)^3}
      \braket{f}{(L_B+|q|^2)^{-s}h}.
    \label{eq:brane-restriction-quadratic-form}
\end{equation}
Equations~\eqref{eq:brane-restriction-operator-integral}--\eqref{eq:brane-restriction-quadratic-form} are the rigorous meaning of ``brane restriction'' in this note.  No point-evaluation map on $L^2(\R^3)$ is used.

\section{Restriction of an ambient Riesz mediator to a codimension-three brane}
\label{sec:brane-restriction}

The first result is the elementary but central observation: in codimension three, the brane-to-brane restriction of $L_M^{-5/2}$ is proportional to $L_B^{-1}$.  Thus a fractional ambient mediator induces the brane Green operator determined by $L_B$.  In flat three-dimensional cases this Green operator has inverse-distance scaling, but the statement below is an operator identity rather than a physical identification.

\begin{lemma}[Transverse momentum integral]
\label{lem:transverse-integral}
Let $m\in\mathbb{N}$, let $s>m/2$, and let $\lambda>0$.  Then
\begin{equation}
    \int_{\R^m}\frac{\dd^m q}{(2\pi)^m}\,\frac{1}{(\lambda+|q|^2)^s}
    =
    \frac{1}{(4\pi)^{m/2}}\frac{\Gamma(s-m/2)}{\Gamma(s)}\lambda^{m/2-s}.
    \label{eq:transverse-integral-general}
\end{equation}
In particular, for $m=3$ and $s=5/2$,
\begin{equation}
    \int_{\R^3}\frac{\dd^3 q}{(2\pi)^3}\,\frac{1}{(\lambda+|q|^2)^{5/2}}
    =\frac{1}{6\pi^2}\frac{1}{\lambda}.
    \label{eq:transverse-integral-3-5half}
\end{equation}
\end{lemma}

\begin{proof}
Use the Schwinger representation
\[
    (\lambda+|q|^2)^{-s}
    =\frac{1}{\Gamma(s)}\int_0^\infty t^{s-1}e^{-t(\lambda+|q|^2)}\dd t.
\]
Then
\[
    \int_{\R^m}\frac{\dd^mq}{(2\pi)^m}e^{-t|q|^2}
    =(4\pi t)^{-m/2}.
\]
Substitution gives
\[
    \frac{1}{(4\pi)^{m/2}\Gamma(s)}\int_0^\infty t^{s-1-m/2}e^{-t\lambda}\dd t
    =\frac{1}{(4\pi)^{m/2}}\frac{\Gamma(s-m/2)}{\Gamma(s)}\lambda^{m/2-s}.
\]
For $m=3$, $s=5/2$, this yields
\[
    (4\pi)^{-3/2}\frac{\Gamma(1)}{\Gamma(5/2)}\lambda^{-1}
    =\frac{1}{6\pi^2}\lambda^{-1}.
\]
\end{proof}

\begin{theorem}[Brane restriction of the codimension-three Riesz mediator]
\label{thm:brane-restriction}
Let $M=B\times\R^3_\perp$ and let
\[
    L_M=L_B\otimes\Id+\Id\otimes(-\Delta_\perp).
\]
For $s>3/2$, the brane-restricted mediator defined by \eqref{eq:brane-restriction-operator-integral} is the spectral multiplier
\begin{equation}
    V_{B,s}=\kappa_s C_{3,s} L_B^{3/2-s},
    \qquad
    C_{3,s}:=\frac{1}{(4\pi)^{3/2}}\frac{\Gamma(s-3/2)}{\Gamma(s)}.
    \label{eq:VB-general-s}
\end{equation}
In particular, writing $\kappa:=\kappa_{5/2}$,
\begin{equation}
    V_{B,5/2}
    =\frac{\kappa}{6\pi^2}L_B^{-1}.
    \label{eq:brane-restricted-Lminusfivehalf}
\end{equation}
Thus, after setting
\begin{equation}
    g:=\frac{\kappa}{6\pi^2},
    \label{eq:g-kappa-relation}
\end{equation}
this brane restriction is exactly
\begin{equation}
    V_B=gL_B^{-1}.
    \label{eq:VB-gLBminus1}
\end{equation}
\end{theorem}

\begin{proof}
By \eqref{eq:brane-restriction-operator-integral},
\[
    V_{B,s}
    =\kappa_s\int_{\R^3}\frac{\dd^3q}{(2\pi)^3}(L_B+|q|^2)^{-s},
\]
where the integral converges in operator norm.  The spectral theorem therefore permits the transverse integral to be evaluated pointwise on the spectrum of $L_B$:
\[
    \int_{\R^3}\frac{\dd^3q}{(2\pi)^3}(L_B+|q|^2)^{-s}
    =\left[\lambda\mapsto
      \int_{\R^3}\frac{\dd^3q}{(2\pi)^3}(\lambda+|q|^2)^{-s}
      \right](L_B).
\]
Lemma~\ref{lem:transverse-integral} gives
\[
    \int_{\R^3}\frac{\dd^3q}{(2\pi)^3}(\lambda+|q|^2)^{-s}
    =C_{3,s}\lambda^{3/2-s}.
\]
Hence
\[
    V_{B,s}=\kappa_s C_{3,s}L_B^{3/2-s}.
\]
For $s=5/2$, $C_{3,5/2}=1/(6\pi^2)$, which gives \eqref{eq:brane-restricted-Lminusfivehalf}.
\end{proof}

\begin{corollary}[General codimension]
\label{cor:general-codimension}
For an ambient product $B\times\R^m_\perp$ and mediator $(L_B-\Delta_\perp)^{-s}$ with $s>m/2$, the brane restriction is proportional to
\[
    L_B^{m/2-s}.
\]
Consequently, the induced brane operator is proportional to $L_B^{-1}$ precisely when
\[
    s=1+\frac{m}{2}.
\]
For $m=3$ this condition gives $s=5/2$.
\end{corollary}

\begin{remark}[Fractional ambient mediator]
The operator $L_M^{-5/2}$ is not the ordinary Green operator $L_M^{-1}$ of a six-dimensional Laplacian.  It is a fractional Riesz-type mediator.  The result above says that this particular fractional ambient mediator induces a brane Green operator proportional to $L_B^{-1}$.  The inverse-distance behavior that appears in flat three-dimensional examples comes from the induced brane Green kernel, not from assuming an ordinary six-dimensional Green kernel.
\end{remark}

\section{Dimensional reduction, canonical restriction, and the critical Riesz exponent}
\label{sec:dimensional-reduction}

The codimension-three restriction formula above is a special case of a
general reduction identity for product geometries.  In this section the
word ``reduction'' has a precise spectral meaning: the transverse continuum
of momenta is integrated out, producing an effective spectral multiplier on
the brane Hilbert space.  No additional geometric or dynamical structure is
used in the reduction.

We first record the reduction in a form that does not rely on informal
point restriction in the transverse \(L^2\)-factor.  Let \(m\geq1\), and
let
\[
    H_M:=H_B\otimes L^2(\R^m_y),
    \qquad
    L_M:=L_B\otimes I+I\otimes(-\Delta_y).
\]
Here \(y\in\R^m\) denotes the transverse variable.  Let
\(\eta\in C_c^\infty(\R^m)\) satisfy
\[
    \int_{\R^m}\eta(y)\,\dd y=1,
\]
and set
\[
    \eta_\varepsilon(y):=\varepsilon^{-m}\eta(y/\varepsilon),
    \qquad \varepsilon>0.
\]
For \(J\in H_B\), define the transversely smeared brane source
\[
    R_\varepsilon J:=J\otimes\eta_\varepsilon\in H_M.
\]
The family \(R_\varepsilon J\) represents a brane-supported source only in
the transverse distributional sense; no convergence in \(H_M\) is asserted.

\begin{proposition}[Canonical transverse restriction]
\label{prop:canonical-transverse-restriction}
Assume \(L_B\geq \ell_B I\) with \(\ell_B>0\).  Let \(s>m/2\).  For every
\(J,K\in H_B\),
\begin{equation}
    \lim_{\varepsilon\to0^+}
    \big\langle R_\varepsilon J,
    L_M^{-s}R_\varepsilon K\big\rangle_{H_M}
    =
    \big\langle J,T_{m,s}(L_B)K\big\rangle_{H_B},
    \label{eq:canonical-restriction-limit}
\end{equation}
where
\begin{equation}
    T_{m,s}(L_B)
    :=
    \int_{\R^m}
    \frac{\dd^m q}{(2\pi)^m}
    (L_B+|q|^2)^{-s}.
    \label{eq:Tms-definition}
\end{equation}
The limit is independent of the choice of normalized mollifier
\(\eta\).
\end{proposition}

\begin{proof}
Use the transverse Fourier transform normalized so that Plancherel gives
the measure \((2\pi)^{-m}\dd^m q\), and so that \(-\Delta_y\) corresponds
to multiplication by \(|q|^2\).  Since
\[
    \widehat{\eta_\varepsilon}(q)=\widehat{\eta}(\varepsilon q),
\]
the direct-integral representation of \(L_M^{-s}\) gives
\begin{align}
    \big\langle R_\varepsilon J,
    L_M^{-s}R_\varepsilon K\big\rangle_{H_M}
    &=
    \int_{\R^m}
    \frac{\dd^m q}{(2\pi)^m}
    \left|\widehat{\eta}(\varepsilon q)\right|^2
    \big\langle J,(L_B+|q|^2)^{-s}K\big\rangle_{H_B}.
    \label{eq:mollified-restriction-integral}
\end{align}
Because \(\eta\in C_c^\infty(\R^m)\), its Fourier transform is bounded.
Moreover,
\[
    \left|
    \big\langle J,(L_B+|q|^2)^{-s}K\big\rangle_{H_B}
    \right|
    \leq
    \|J\|_{H_B}\|K\|_{H_B}(\ell_B+|q|^2)^{-s}.
\]
The majorant \((\ell_B+|q|^2)^{-s}\) is integrable over \(\R^m\) exactly
when \(s>m/2\).  Finally,
\[
    \widehat{\eta}(0)=\int_{\R^m}\eta(y)\,\dd y=1.
\]
Dominated convergence applied to
\eqref{eq:mollified-restriction-integral} proves
\eqref{eq:canonical-restriction-limit}.  Since only the normalization
\(\widehat{\eta}(0)=1\) enters the limiting value, the limit is independent
of the particular mollifier.
\end{proof}

The operator \(T_{m,s}(L_B)\) admits an equivalent heat-kernel form.  This
form isolates the scaling responsible for the exponent \(5/2\) in
codimension three.

\begin{proposition}[Transverse reduction in heat-kernel form]
\label{prop:Tms-heat-form}
Assume \(L_B\geq \ell_B I\) with \(\ell_B>0\).  Let \(m\geq1\) and
\(s>m/2\).  Then \(T_{m,s}(L_B)\) exists as an operator-norm Bochner
integral and satisfies
\begin{equation}
    T_{m,s}(L_B)
    =
    \frac{1}{(4\pi)^{m/2}\Gamma(s)}
    \int_0^\infty
    t^{s-1-m/2}e^{-tL_B}\,\dd t .
    \label{eq:Tms-heat-form}
\end{equation}
Equivalently,
\begin{equation}
    T_{m,s}(L_B)
    =
    \frac{1}{(4\pi)^{m/2}}
    \frac{\Gamma(s-m/2)}{\Gamma(s)}
    L_B^{m/2-s}.
    \label{eq:Tms-spectral-form}
\end{equation}
\end{proposition}

\begin{proof}
Since \(L_B\geq \ell_B I\),
\[
    \left\|(L_B+|q|^2)^{-s}\right\|
    \leq
    (\ell_B+|q|^2)^{-s}.
\]
The right-hand side is integrable over \(\R^m\) precisely when \(s>m/2\).
Thus \eqref{eq:Tms-definition} is an operator-norm Bochner integral.

For each fixed \(q\), the Schwinger representation for positive
self-adjoint operators gives
\[
    (L_B+|q|^2)^{-s}
    =
    \frac{1}{\Gamma(s)}
    \int_0^\infty
    t^{s-1}e^{-t(L_B+|q|^2)}\,\dd t .
\]
The preceding norm estimate justifies Fubini in operator norm.  Therefore
\begin{align}
    T_{m,s}(L_B)
    &=
    \frac{1}{\Gamma(s)}
    \int_0^\infty
    t^{s-1}e^{-tL_B}
    \left(
        \int_{\R^m}
        \frac{\dd^m q}{(2\pi)^m}
        e^{-t|q|^2}
    \right)
    \dd t                                      \notag\\
    &=
    \frac{1}{(4\pi)^{m/2}\Gamma(s)}
    \int_0^\infty
    t^{s-1-m/2}e^{-tL_B}\,\dd t .
\end{align}
This proves \eqref{eq:Tms-heat-form}.  The same integral is convergent in
operator norm because
\[
    \left\|t^{s-1-m/2}e^{-tL_B}\right\|
    \leq
    t^{s-1-m/2}e^{-t\ell_B},
\]
which is integrable at \(t=0\) when \(s>m/2\), and integrable at infinity
because \(\ell_B>0\).

Applying the spectral theorem to \eqref{eq:Tms-heat-form}, for each
spectral value \(\lambda\geq\ell_B\) one obtains
\[
    \frac{1}{(4\pi)^{m/2}\Gamma(s)}
    \int_0^\infty
    t^{s-1-m/2}e^{-t\lambda}\,\dd t
    =
    \frac{1}{(4\pi)^{m/2}}
    \frac{\Gamma(s-m/2)}{\Gamma(s)}
    \lambda^{m/2-s}.
\]
This proves \eqref{eq:Tms-spectral-form}.
\end{proof}

\begin{proposition}[Critical Riesz exponent for an induced brane Green operator]
\label{prop:critical-riesz-exponent}
For the reduced multiplier in \eqref{eq:Tms-spectral-form}, the Green
power \(L_B^{-1}\) is obtained exactly at
\begin{equation}
    s=s_\ast(m):=1+\frac m2 .
    \label{eq:critical-riesz-exponent}
\end{equation}
At this exponent,
\begin{equation}
    T_{m,s_\ast}(L_B)
    =
    \frac{1}{(4\pi)^{m/2}\Gamma(1+m/2)}
    L_B^{-1}.
    \label{eq:T-critical-general}
\end{equation}
In codimension three,
\begin{equation}
    s_\ast(3)=\frac52,
    \qquad
    T_{3,5/2}(L_B)
    =
    \frac{1}{6\pi^2}L_B^{-1}.
    \label{eq:T-critical-three}
\end{equation}
\end{proposition}

\begin{proof}
By Proposition~\ref{prop:Tms-heat-form},
\[
    T_{m,s}(L_B)
    =
    C_{m,s}L_B^{m/2-s},
    \qquad
    C_{m,s}
    =
    \frac{1}{(4\pi)^{m/2}}
    \frac{\Gamma(s-m/2)}{\Gamma(s)} .
\]
The reduced spectral multiplier has the Green power \(\lambda^{-1}\)
precisely when
\[
    \lambda^{m/2-s}=\lambda^{-1}
\]
as a function of \(\lambda>0\).  Hence
\[
    m/2-s=-1,
    \qquad\text{or equivalently}\qquad
    s=1+\frac m2 .
\]
Substitution gives
\[
    C_{m,s_\ast}
    =
    \frac{1}{(4\pi)^{m/2}}
    \frac{\Gamma(1)}{\Gamma(1+m/2)}
    =
    \frac{1}{(4\pi)^{m/2}\Gamma(1+m/2)}.
\]
This proves \eqref{eq:T-critical-general}.  For \(m=3\),
\(\Gamma(5/2)=3\sqrt{\pi}/4\), so
\[
    \frac{1}{(4\pi)^{3/2}\Gamma(5/2)}
    =
    \frac{1}{6\pi^2}.
\]
This proves \eqref{eq:T-critical-three}.
\end{proof}

\begin{remark}[Why the exponent \(5/2\) appears]
The heat-kernel representation \eqref{eq:Tms-heat-form} makes the exponent
transparent.  The Schwinger parameter for the Riesz power contributes
\(t^{s-1}\), while the \(m\)-dimensional transverse heat kernel contributes
\(t^{-m/2}\).  At
\[
    s=1+\frac m2,
\]
these powers combine to \(t^0\).  Therefore
\[
    T_{m,s_\ast}(L_B)
    =
    \frac{1}{(4\pi)^{m/2}\Gamma(1+m/2)}
    \int_0^\infty e^{-tL_B}\,\dd t
    =
    \frac{1}{(4\pi)^{m/2}\Gamma(1+m/2)}
    L_B^{-1}.
\]
Thus \(s=5/2\) is the codimension-three value selected by transverse
heat-kernel scaling.  The word ``critical'' is used only in this spectral
scaling sense.
\end{remark}

There is also an integer-order auxiliary representation of the fractional
operator \(L_M^{-5/2}\).  

We also record a related integer-order auxiliary identity: the fractional
operator \(L_M^{-5/2}\) can be represented, up to an explicit constant, as a
hypersurface reduction of a cubic resolvent in one additional auxiliary
direction.

\begin{proposition}[Auxiliary integer-order lift]
\label{prop:auxiliary-integer-lift}
Let \(A\geq \ell I\), \(\ell>0\), be a positive self-adjoint operator on a
Hilbert space \(H\).  Define
\[
    \widetilde A:=A\otimes I+I\otimes(-\partial_z^2)
\]
on \(H\otimes L^2(\R_z)\).  Let \(\rho\in C_c^\infty(\R)\) satisfy
\[
    \int_{\R}\rho(z)\,\dd z=1,
\]
and set
\[
    \rho_\delta(z):=\delta^{-1}\rho(z/\delta),
    \qquad \delta>0.
\]
For \(F\in H\), define
\[
    S_\delta F:=F\otimes\rho_\delta\in H\otimes L^2(\R_z).
\]
Then, for all \(F,G\in H\),
\begin{equation}
    \lim_{\delta\to0^+}
    \big\langle S_\delta F,\widetilde A^{-3}S_\delta G
    \big\rangle_{H\otimes L^2(\R)}
    =
    \frac{3}{16}
    \big\langle F,A^{-5/2}G\big\rangle_H .
    \label{eq:auxiliary-lift-limit}
\end{equation}
Equivalently, in canonical hypersurface-restriction notation,
\begin{equation}
    R_z\widetilde A^{-3}R_z^\ast
    =
    \frac{3}{16}A^{-5/2}.
    \label{eq:auxiliary-lift-formal}
\end{equation}
\end{proposition}

\begin{proof}
The proof is the one-dimensional instance of
Proposition~\ref{prop:canonical-transverse-restriction}, with \(m=1\),
\(s=3\), and \(L_B\) replaced by \(A\).  Explicitly,
\[
    \lim_{\delta\to0^+}
    \big\langle S_\delta F,\widetilde A^{-3}S_\delta G
    \big\rangle
    =
    \int_{\R}\frac{\dd p}{2\pi}
    \big\langle F,(A+p^2)^{-3}G\big\rangle_H .
\]
By the functional calculus,
\[
    \int_{\R}\frac{\dd p}{2\pi}(A+p^2)^{-3}
    =
    \frac{1}{(4\pi)^{1/2}}
    \frac{\Gamma(3-1/2)}{\Gamma(3)}
    A^{1/2-3}.
\]
Since
\[
    \frac{1}{(4\pi)^{1/2}}
    \frac{\Gamma(5/2)}{\Gamma(3)}
    =
    \frac{1}{2\sqrt{\pi}}\cdot
    \frac{3\sqrt{\pi}/4}{2}
    =
    \frac{3}{16},
\]
we obtain
\[
    \int_{\R}\frac{\dd p}{2\pi}(A+p^2)^{-3}
    =
    \frac{3}{16}A^{-5/2}.
\]
This proves \eqref{eq:auxiliary-lift-limit}.
\end{proof}

This auxiliary lift is not used in the proof of the stochastic trace
identity.  Its role is only to record that the fractional inverse
\(L_M^{-5/2}\) may be obtained, with an explicit constant, as a hypersurface
reduction of an integer-power resolvent in one additional auxiliary
direction.

Applying Proposition~\ref{prop:auxiliary-integer-lift} with \(A=L_M\)
gives
\begin{equation}
    R_z(L_M-\partial_z^2)^{-3}R_z^\ast
    =
    \frac{3}{16}L_M^{-5/2},
    \label{eq:LM-fivehalf-from-cubic}
\end{equation}
again in the canonical mollified sense.  Thus the codimension-three Riesz
mediator \(L_M^{-5/2}\) can be obtained as the hypersurface reduction of an
integer-power cubic resolvent in one auxiliary direction.

Combining the auxiliary \(z\)-reduction with the transverse
codimension-three reduction gives a single four-dimensional transverse
calculation:
\begin{equation}
    \int_{\R^3}\frac{\dd^3 q}{(2\pi)^3}
    \int_{\R}\frac{\dd p}{2\pi}
    (L_B+|q|^2+p^2)^{-3}
    =
    \frac{1}{32\pi^2}L_B^{-1}.
    \label{eq:combined-four-dimensional-reduction}
\end{equation}
Equivalently,
\begin{equation}
    \frac{16}{3}
    \left(
    \int_{\R}\frac{\dd p}{2\pi}
    (L_M+p^2)^{-3}
    \right)
    =
    L_M^{-5/2},
    \qquad
    \int_{\R^3}\frac{\dd^3 q}{(2\pi)^3}
    L_M^{-5/2}
    =
    \frac{1}{6\pi^2}L_B^{-1},
    \label{eq:two-step-reduction-chain}
\end{equation}
with the product of constants
\[
    \frac{3}{16}\cdot\frac{1}{6\pi^2}
    =
    \frac{1}{32\pi^2}.
\]

\begin{remark}[Integer-order quadratic form]
When \(L_M\) is a local second-order elliptic operator, the auxiliary
operator
\[
    \widetilde L:=L_M-\partial_z^2
\]
is again second order on the product with the auxiliary line, and
\(\widetilde L^3\) is an integer-order sixth-order differential operator.
At finite spectral and transverse-momentum cutoffs, the quadratic form
\[
    Q_{\mathrm{aux}}[\psi]
    =
    \big\langle \psi,\widetilde L^3\psi\big\rangle
\]
has covariance \(\widetilde L^{-3}\).  The canonical hypersurface
restriction in Proposition~\ref{prop:auxiliary-integer-lift} then produces
the fractional covariance factor \(L_M^{-5/2}\), with the explicit constant
\(3/16\).  This is the sense in which the fractional mediator used above is
related to an integer-order auxiliary resolvent.
\end{remark}

\begin{remark}[Role in the present construction]
The results of this section give three precise facts used by the later
stochastic trace representation.  First, brane restriction of
\(L_M^{-s}\) is defined as a mollifier-independent transverse distributional
limit when \(s>m/2\).  Second, after this restriction, the reduced brane
operator is the spectral multiplier \(T_{m,s}(L_B)\).  Third, the value
\(s=1+m/2\) is exactly the value for which this multiplier is proportional
to the ordinary Green operator \(L_B^{-1}\).  For the codimension-three
case used in the paper, this gives \(s=5/2\) and
\[
    T_{3,5/2}(L_B)=\frac{1}{6\pi^2}L_B^{-1}.
\]
The stochastic source construction below takes this induced brane Green
operator as its quadratic kernel.
\end{remark}

\section{Gaussian generalized sources and a prescribed-covariance trace representation}
\label{sec:stochastic}

The second ingredient is a Gaussian generalized scalar source whose covariance is chosen to compensate the inverse power of $L_B$ in the brane mediator.  This is the point at which the construction becomes a representation theorem for a spectral trace.  The source is defined as a generalized Gaussian field in the Hilbert scale of $L_B$; no ordinary $H_B$-valued source is assumed before heat regularization.

\subsection{White noise, negative Sobolev scale, and fractional differentiation}

For $r\in\R$, let $\mathcal H_B^r$ denote the Hilbert scale generated by $L_B$; in the compact-resolvent case,
\[
    \norm{v}_{\mathcal H_B^r}^2
    =\sum_{j\geq1}\lambda_j^r\,|v_j|^2,
    \qquad
    v_j=\braket{u_j}{v}.
\]
Thus $\mathcal H_B^r=\Dom(L_B^{r/2})$ for $r\geq0$, while $\mathcal H_B^{-r}$ is the corresponding negative Sobolev-type completion.  Let $\xi$ denote real or complex Gaussian white noise on $H_B$, formally written as
\begin{equation}
    \xi=\sum_{j\geq1}\xi_j u_j,
    \qquad
    \E[\xi_i\overline{\xi_j}]=\delta_{ij}.
    \label{eq:white-noise-expansion}
\end{equation}
This series does not converge in $H_B$ in general.  It is first of all a cylindrical Gaussian random distribution, or equivalently an isonormal Gaussian process on $H_B$.  If $\Tr(L_B^{-r})<\infty$, then the same series converges in $L^2(\Omega;\mathcal H_B^{-r})$ and almost surely in $\mathcal H_B^{-r}$; for a standard compact elliptic operator on a three-dimensional brane this holds for every $r>3/2$.

Let $g>0$ be the brane normalization constant in \eqref{eq:VB-gLBminus1}.  The unregularized expression
\[
    \sigma:=\left(\frac{\hbar c}{g}\right)^{1/2}L_B^{3/4}\xi
\]
is not an $H_B$-valued random variable.  It is a generalized Gaussian field, defined on test vectors $\varphi\in\mathcal H_B^{3/2}=\Dom(L_B^{3/4})$ by
\[
    \sigma(\varphi)
    :=\left(\frac{\hbar c}{g}\right)^{1/2}\xi(L_B^{3/4}\varphi).
\]
Its covariance is the quadratic form
\[
    \E[\sigma(\varphi)\overline{\sigma(\psi)}]
    =\frac{\hbar c}{g}\braket{L_B^{3/4}\varphi}{L_B^{3/4}\psi},
    \qquad \varphi,\psi\in\mathcal H_B^{3/2}.
\]
Equivalently, its formal covariance is $(\hbar c/g)L_B^{3/2}$.  Under stronger trace assumptions, for example $\Tr(L_B^{3/2-r})<\infty$, this generalized field can be realized as an $\mathcal H_B^{-r}$-valued random variable.  In the usual compact elliptic three-dimensional case this is true for every $r>3$.

For $\tau>0$, define the heat-regularized source
\begin{equation}
    \sigma_\tau
    :=\left(\frac{\hbar c}{g}\right)^{1/2}L_B^{3/4}e^{-\tau L_B/2}\xi.
    \label{eq:sigma-tau}
\end{equation}
This is an ordinary $H_B$-valued Gaussian random variable.  Indeed,
\[
    \E\norm{\sigma_\tau}_{H_B}^2
    =\frac{\hbar c}{g}\Tr(L_B^{3/2}e^{-\tau L_B})<\infty,
\]
because $x^{3/2}e^{-\tau x}\leq C_\tau e^{-\tau x/2}$ and $\Tr(e^{-\tau L_B/2})<\infty$.  In components,
\begin{equation}
    \sigma_{\tau,j}:=\braket{u_j}{\sigma_\tau}
    =\left(\frac{\hbar c}{g}\right)^{1/2}\lambda_j^{3/4}e^{-\tau\lambda_j/2}\xi_j.
    \label{eq:sigma-components}
\end{equation}
Therefore
\begin{equation}
    \E[\sigma_\tau\otimes\sigma_\tau^*]
    =\frac{\hbar c}{g}L_B^{3/2}e^{-\tau L_B},
    \label{eq:sigma-covariance}
\end{equation}
where the covariance operator on the right is trace class for every $\tau>0$.

\begin{remark}[Meaning of the exponent \(3/4\)]
The exponent \(3/4\) is fixed by the prescribed source covariance in
\eqref{eq:sigma-covariance}.  Pairing the covariance
\[
    \frac{\hbar c}{g}L_B^{3/2}e^{-\tau L_B}
\]
with the brane Green kernel \(gL_B^{-1}\) gives the multiplier
\[
    \hbar c\,L_B^{1/2}e^{-\tau L_B},
\]
which is the heat-regularized scalar Casimir trace multiplier.  Thus the
covariance is input data for the representation, not a consequence derived
from an independent microscopic source model.
\end{remark}

\subsection{Regulated stochastic interaction energy}

Define the regulated brane interaction energy
\begin{equation}
    U_\tau
    :=\frac12\braket{\sigma_\tau}{gL_B^{-1}\sigma_\tau}.
    \label{eq:regulated-U}
\end{equation}
For $\tau>0$ this quadratic form is well-defined almost surely because $\sigma_\tau\in H_B$ and $L_B^{-1}$ is bounded by Assumption~\ref{ass:L_B}.  It has finite expectation because
\[
    \E[U_\tau]
    =\frac{\hbar c}{2}\sum_j\lambda_j^{1/2}e^{-\tau\lambda_j},
\]
which is finite under the heat-trace hypotheses.  The following theorem is the corresponding Gaussian quadratic-form identity.

The covariance in \eqref{eq:sigma-covariance} is part of the data of the
representation.  The following theorem records the quadratic-form identity
that follows from this prescribed covariance; it is not a derivation of the
covariance from an independent microscopic model.

\begin{theorem}[Quadratic-form representation under the prescribed covariance]
\label{thm:stochastic-trace}
For every $\tau>0$,
\begin{equation}
    \E[U_\tau]
    =\frac{\hbar c}{2}\Tr\left(L_B^{1/2}e^{-\tau L_B}\right).
    \label{eq:stochastic-trace-identity}
\end{equation}
\end{theorem}

\begin{proof}
Using \eqref{eq:sigma-components} in \eqref{eq:regulated-U},
\[
    U_\tau
    =\frac{g}{2}\sum_j\frac{|\sigma_{\tau,j}|^2}{\lambda_j}.
\]
Taking expectation gives
\[
    \E[U_\tau]
    =\frac{g}{2}\sum_j\frac{1}{\lambda_j}
      \E[|\sigma_{\tau,j}|^2].
\]
By \eqref{eq:sigma-components},
\[
    \E[|\sigma_{\tau,j}|^2]
    =\frac{\hbar c}{g}\lambda_j^{3/2}e^{-\tau\lambda_j}.
\]
Therefore
\[
    \E[U_\tau]
    =\frac{g}{2}\sum_j\frac{1}{\lambda_j}\frac{\hbar c}{g}\lambda_j^{3/2}e^{-\tau\lambda_j}
    =\frac{\hbar c}{2}\sum_j\lambda_j^{1/2}e^{-\tau\lambda_j},
\]
which is \eqref{eq:stochastic-trace-identity}.
\end{proof}

\begin{remark}[Positivity and negative finite parts]
For each fixed $\tau>0$, the random variable $U_\tau$ is non-negative when $gL_B^{-1}$ is positive.  The renormalized finite part obtained as $\tau\to0^+$ need not be positive.  This is not a contradiction: the finite part is obtained after subtracting divergent local terms, and finite remainders of positive divergent quantities may be negative.  The same phenomenon occurs in standard Casimir calculations.
\end{remark}

\section{Renormalized finite parts}
\label{sec:renormalized}

The identity in Theorem~\ref{thm:stochastic-trace} is a regulated identity.  The corresponding scalar finite part is obtained by subtracting heat-kernel divergences.  This section records the finite-part statement used throughout the rest of the note.

\subsection{Finite-part prescription}

Let
\begin{equation}
    \mathcal{E}_\tau(L_B)
    :=\frac{\hbar c}{2}\Tr\left(L_B^{1/2}e^{-\tau L_B}\right).
    \label{eq:regularized-casimir-trace}
\end{equation}
The regulator used throughout this note is the heat cutoff $e^{-\tau L_B}$.  A finite-part prescription is specified as follows.  Assume that, as $\tau\to0^+$, $\mathcal{E}_\tau(L_B)$ admits an asymptotic expansion of the form
\begin{equation}
    \mathcal{E}_\tau(L_B)
    \sim
    \sum_{\beta\in\mathcal D} a_\beta(L_B)\tau^{-\beta}
    +a_{\log}(L_B)\log(\mu^2\tau)
    +a_0(L_B)
    +\sum_{\eta\in\mathcal P}
      \bigl(b_\eta(L_B)\tau^\eta
      +c_\eta(L_B)\tau^\eta\log(\mu^2\tau)\bigr).
    \label{eq:finite-part-expansion}
\end{equation}
Here $\mathcal D\subset(0,\infty)$ is finite, $\mathcal P\subset(0,\infty)$ indexes terms that vanish as $\tau\to0^+$, $\mu$ is a fixed reference scale used only when a logarithmic term is present, and the coefficients $a_\beta$ and $a_{\log}$ are local heat-kernel counterterms for elliptic boundary problems.  The minimal heat-kernel finite part is
\begin{equation}
    \FP_{\min,\tau\to0^+}\mathcal{E}_\tau(L_B)
    :=\lim_{\tau\to0^+}
    \left[\mathcal{E}_\tau(L_B)
      -\sum_{\beta\in\mathcal D} a_\beta(L_B)\tau^{-\beta}
      -a_{\log}(L_B)\log(\mu^2\tau)
    \right],
    \label{eq:minimal-finite-part}
\end{equation}
when the limit exists.  A general prescription $\mathcal R$ is obtained from \eqref{eq:minimal-finite-part} by adding a specified finite local counterterm $C_{\mathcal R}^{\mathrm{loc}}(L_B)$:
\begin{equation}
    \FP_{\mathcal R,\tau\to0^+}\mathcal{E}_\tau(L_B)
    :=\FP_{\min,\tau\to0^+}\mathcal{E}_\tau(L_B)
      +C_{\mathcal R}^{\mathrm{loc}}(L_B).
    \label{eq:R-finite-part}
\end{equation}
The same subtraction terms, reference scale, and finite local counterterm are applied to the stochastic expectation $\E[U_\tau]$ after identifying it with \eqref{eq:regularized-casimir-trace}.  For the parallel-plate specialization below, the prescription is the standard interaction-energy subtraction: remove the bulk and one-plate self-energy terms and retain the finite separation-dependent part, equivalently the zeta finite part recalled in Appendix~\ref{app:parallel-zeta}.

\begin{definition}[Renormalized scalar Casimir trace]
Given a fixed finite-part prescription $\mathcal{R}$, define
\begin{equation}
    E_{\mathrm{Cas},\mathcal{R}}(L_B)
    :=\FP_{\mathcal{R},\,\tau\to0^+}
    \frac{\hbar c}{2}\Tr\left(L_B^{1/2}e^{-\tau L_B}\right).
    \label{eq:renorm-casimir}
\end{equation}
\end{definition}

\begin{definition}[Renormalized expected stochastic interaction]
With $U_\tau$ as in \eqref{eq:regulated-U}, define
\begin{equation}
    \E[U_{\mathcal{R}}]
    :=\FP_{\mathcal{R},\,\tau\to0^+}\E[U_\tau].
    \label{eq:renorm-stochastic-energy}
\end{equation}
This is a renormalized expectation.  The construction does not require the existence of a renormalized random variable $U_{\mathcal{R}}(\omega)$ configuration-by-configuration.
\end{definition}

\begin{proposition}[Equality of renormalized finite parts]
\label{prop:renormalized-equality}
For every finite-part prescription $\mathcal{R}$ applied identically to both sides,
\begin{equation}
    \E[U_{\mathcal{R}}]=E_{\mathrm{Cas},\mathcal{R}}(L_B).
    \label{eq:renormalized-equality}
\end{equation}
\end{proposition}

\begin{proof}
The regulated identity \eqref{eq:stochastic-trace-identity} holds for every $\tau>0$.  Applying the same linear finite-part operation to both sides gives \eqref{eq:renormalized-equality}.
\end{proof}

\begin{remark}[Scheme dependence]
The numerical value of a finite part may change if finite local counterterms are changed.  Proposition~\ref{prop:renormalized-equality} is not a claim of scheme independence.  It is a claim that, under an identical scheme, the stochastic quadratic form and the scalar Casimir trace have the same finite part.
\end{remark}

\section{Scalar parallel-plate benchmark}
\label{sec:parallel-plates}

We now specialize the abstract result to the standard scalar parallel-plate geometry.  The purpose of this section is not to rederive the electromagnetic Casimir effect or to introduce a new physical plate model.  It is to evaluate the abstract scalar trace identity in the standard scalar Dirichlet parallel-plate geometry, thereby fixing the normalization used in the dimensionless comparison below.

\subsection{Finite lateral box}

Let
\[
    B_{L,a}:=\mathbb{T}_L^2\times[0,a],
\]
where $\mathbb{T}_L^2$ is a flat two-torus of side length $L$ and area $A=L^2$.  Let
\[
    L_B=-\Delta_{\mathbb{T}_L^2}-\partial_z^2
\]
with Dirichlet boundary conditions at $z=0$ and $z=a$.  The eigenvalues are
\begin{equation}
    \lambda_{\mathbf{m},n}
    =\left(\frac{2\pi}{L}\right)^2|\mathbf{m}|^2+\left(\frac{\pi n}{a}\right)^2,
    \qquad
    \mathbf{m}\in\mathbb{Z}^2,
    \quad n\in\mathbb{N}.
    \label{eq:plate-eigenvalues}
\end{equation}
The regulated stochastic identity gives
\begin{equation}
    \E[U_\tau]
    =\frac{\hbar c}{2}\sum_{\mathbf{m}\in\mathbb{Z}^2}\sum_{n=1}^\infty
    \lambda_{\mathbf{m},n}^{1/2}e^{-\tau\lambda_{\mathbf{m},n}}.
    \label{eq:finite-box-trace}
\end{equation}

\subsection{Large-area scalar result}

Passing to the large-area limit $L\to\infty$, the scalar Dirichlet finite part per unit area is
\begin{equation}
    \frac{E_{\mathrm{Cas}}^{(1)}}{A}
    =-\frac{\pi^2\hbar c}{1440a^3}.
    \label{eq:scalar-casimir-per-area}
\end{equation}
The superscript $(1)$ denotes one scalar channel.  A derivation by zeta regularization is recalled in Appendix~\ref{app:parallel-zeta}.  For the scalar Dirichlet plate geometry, the heat-kernel finite part used in the main text and the zeta-regularized finite separation-dependent term give the same coefficient after subtraction of the bulk and one-plate self-energy contributions.  We use the zeta calculation only as a compact way to recall this standard finite part.

Combining \eqref{eq:scalar-casimir-per-area} with Proposition~\ref{prop:renormalized-equality} gives
\begin{equation}
    \lim_{L\to\infty}\frac{1}{A}\E[U_{\mathcal{R}}^{(1)}]
    =-\frac{\pi^2\hbar c}{1440a^3},
    \label{eq:expected-stochastic-per-area-scalar}
\end{equation}
where $\mathcal{R}$ denotes the usual parallel-plate subtraction.

If $N$ independent scalar channels are included, then
\begin{equation}
    \lim_{L\to\infty}\frac{1}{A}\E[U_{\mathcal{R}}^{(N)}]
    =-\frac{N\pi^2\hbar c}{1440a^3}.
    \label{eq:expected-stochastic-per-area-N}
\end{equation}
For $N=2$, this gives the two-channel scalar plate coefficient
\begin{equation}
    \frac{E_{\mathrm{Cas}}^{(2)}}{A}
    =-\frac{\pi^2\hbar c}{720a^3}.
    \label{eq:two-channel-plate-energy}
\end{equation}

\begin{remark}[Scalar doubling versus vector boundary problems]
Equation~\eqref{eq:two-channel-plate-energy} is a scalar two-channel result.  It is obtained by taking two independent copies of the scalar trace identity.  A vector-field boundary-value problem would require the corresponding vector operator and transverse boundary structure.  The scalar theorem above is independent of such an extension.
\end{remark}

\subsection{The \texorpdfstring{$A=n^2a^2$}{A = n squared a squared} normalization}

For the reference-cell Green-energy normalization below, set
\[
    A=n^2a^2.
\]
Then \eqref{eq:expected-stochastic-per-area-N} becomes
\begin{equation}
    \E[U_{\mathcal{R}}^{(N)}]
    =-\frac{n^2}{a}\frac{N\pi^2\hbar c}{1440}
    \quad\text{in the large-area limit.}
    \label{eq:expected-stochastic-n-a}
\end{equation}
For $N=2$,
\begin{equation}
    \E[U_{\mathcal{R}}^{(2)}]
    =-\frac{n^2}{a}\frac{\pi^2\hbar c}{720}.
    \label{eq:expected-stochastic-two-channel-na}
\end{equation}

\section{Reference Green-energy normalization}
\label{sec:reference-green-normalization}
\label{sec:box-integral}

The finite-part plate energy in Section~\ref{sec:parallel-plates} has the
large-area form
\begin{equation}
    \E[U_{\mathcal R}^{(N)}]
    =
    -\frac{n^2}{a}\frac{N\pi^2\hbar c}{1440}
    \label{eq:reference-plate-energy-na}
\end{equation}
when the lateral area is written as \(A=n^2a^2\).  Thus the plate separation
\(a\) fixes both the transverse Dirichlet scale of the parallel-plate
operator and the lateral area scale \(a^2\) of a single reference cell.

The question addressed in this section is separate from the proof of the
stochastic trace identity: once the scalar finite part has been written as a
Green-energy expectation, choose a deterministic flat \(L_0^{-1}\) reference
energy at the same plate scale and determine the plate-compatible
rectangular cell for which the resulting calibration coefficient is
extremal.

This section defines a deterministic flat Green-energy normalization adapted
to that plate scale.  The selection statements below are conditional on the
plate-compatible rectangular product class
\[
    R_{\ell_1,\ell_2,a}
    =
    [0,\ell_1]\times[0,\ell_2]\times[0,a],
    \qquad
    \ell_1\ell_2=a^2.
\]
No optimality over arbitrary unit-volume subsets of \(\R^3\) is asserted.
Within this rectangular product class, the cubical cell is selected by two
spectral criteria internal to the plate geometry: it uniquely maximizes the
free lateral spectral gap, and it uniquely minimizes the first
shape-dependent artificial-boundary coefficient in the mixed heat trace.  The
same cubical cell also maximizes the deterministic flat Green energy
\(\Delta(\alpha)\) over the plate-compatible rectangular aspect-ratio family,
and hence minimizes the associated scalar comparison coefficient.

\subsection{Plate-compatible cells and free lateral faces}
\label{sec:plate-compatible-cells}

Let
\begin{equation}
    Q_{\ell_1,\ell_2}:=[0,\ell_1]\times[0,\ell_2],
    \qquad
    R_{\ell_1,\ell_2,a}:=Q_{\ell_1,\ell_2}\times[0,a],
    \label{eq:reference-rectangular-cell}
\end{equation}
with
\begin{equation}
    \ell_1\ell_2=a^2.
    \label{eq:reference-area-constraint}
\end{equation}
Equivalently, write
\begin{equation}
    \ell_1=\alpha a,
    \qquad
    \ell_2=\alpha^{-1}a,
    \qquad
    \alpha>0.
    \label{eq:reference-aspect-parametrization}
\end{equation}
The corresponding unit-volume reference cell is
\begin{equation}
    D_\alpha
    :=
    [0,\alpha]\times[0,\alpha^{-1}]\times[0,1],
    \qquad
    R_{\ell_1,\ell_2,a}=aD_\alpha.
    \label{eq:unit-reference-aspect-cell}
\end{equation}
The cubical cell corresponds to \(\alpha=1\):
\[
    C=[0,1]^3,
    \qquad
    C_a=[0,a]^3.
\]

The lateral cell faces are artificial cuts introduced by the reference-cell
decomposition.  They are not physical conducting plates.  Accordingly, the
lateral spectral scale used below is the one associated with the
unconstrained lateral Dirichlet energy form.

Let
\[
    \mathfrak q_{\parallel,\ell_1,\ell_2}[v]
    :=
    \int_{Q_{\ell_1,\ell_2}}|\nabla_\parallel v|^2\,\dd x\,\dd y,
    \qquad
    \Dom(\mathfrak q_{\parallel,\ell_1,\ell_2})
    =
    H^1(Q_{\ell_1,\ell_2}).
\]
By the representation theorem for closed semibounded quadratic forms, this
form determines the Neumann Laplacian on \(Q_{\ell_1,\ell_2}\).  In this
precise sense, Neumann conditions are the free-boundary realization of the
local lateral energy on an artificial cell face.

For the full rectangular cell define the mixed form
\begin{equation}
    \mathfrak q_{\ell_1,\ell_2,a}[u]
    :=
    \int_{R_{\ell_1,\ell_2,a}}|\nabla u|^2\,\dd x\,\dd y\,\dd z,
    \label{eq:mixed-cell-form}
\end{equation}
with domain
\begin{equation}
    \Dom(\mathfrak q_{\ell_1,\ell_2,a})
    =
    \left\{
        u\in H^1(R_{\ell_1,\ell_2,a}):
        u|_{z=0}=u|_{z=a}=0
        \text{ in the trace sense}
    \right\}.
    \label{eq:mixed-cell-form-domain}
\end{equation}
Thus Dirichlet conditions are imposed only on the physical plate faces, while
no trace constraint is imposed on the artificial lateral faces.

Let
\[
    L_{\ell_1,\ell_2,a}^{\mathrm{cell}}
\]
be the positive self-adjoint operator associated with
\(\mathfrak q_{\ell_1,\ell_2,a}\).

\begin{lemma}[Mixed cell spectrum]
\label{lem:mixed-cell-spectrum}
The operator \(L_{\ell_1,\ell_2,a}^{\mathrm{cell}}\) is the Laplacian on
\(R_{\ell_1,\ell_2,a}\) with Neumann conditions on the lateral faces
\(x=0,\ell_1\), \(y=0,\ell_2\), and Dirichlet conditions on the plate faces
\(z=0,a\).  Its eigenvalues are
\begin{equation}
    \lambda_{m,n,r}^{\mathrm{cell}}
    =
    \frac{\pi^2m^2}{\ell_1^2}
    +
    \frac{\pi^2n^2}{\ell_2^2}
    +
    \frac{\pi^2r^2}{a^2},
    \qquad
    m,n\in\mathbb N_0,\quad r\in\mathbb N.
    \label{eq:mixed-cell-eigenvalues}
\end{equation}
\end{lemma}

\begin{proof}
The one-dimensional quadratic form
\[
    \int_0^\ell |f'(x)|^2\,\dd x,
    \qquad
    \Dom=H^1(0,\ell),
\]
has the Neumann Laplacian as its associated operator, with eigenfunctions
\(\cos(\pi m x/\ell)\), \(m\in\mathbb N_0\), and eigenvalues
\(\pi^2m^2/\ell^2\).  The one-dimensional form
\[
    \int_0^a |g'(z)|^2\,\dd z,
    \qquad
    \Dom=H^1_0(0,a),
\]
has the Dirichlet Laplacian as its associated operator, with eigenfunctions
\(\sin(\pi r z/a)\), \(r\in\mathbb N\), and eigenvalues
\(\pi^2r^2/a^2\).  The form
\(\mathfrak q_{\ell_1,\ell_2,a}\) is the tensor-sum form of these three
one-dimensional forms.  Therefore its associated operator has the separated
eigenbasis
\[
    \cos\!\left(\frac{\pi m x}{\ell_1}\right)
    \cos\!\left(\frac{\pi n y}{\ell_2}\right)
    \sin\!\left(\frac{\pi r z}{a}\right),
    \qquad
    m,n\in\mathbb N_0,\quad r\in\mathbb N,
\]
with eigenvalues \eqref{eq:mixed-cell-eigenvalues}.
\end{proof}

\subsection{Extremal lateral spectral scale}
\label{sec:extremal-lateral-scale}

The first transverse Dirichlet eigenvalue of the plate direction
\(z\in(0,a)\) is
\begin{equation}
    \lambda_\perp(a)
    :=
    \frac{\pi^2}{a^2}.
    \label{eq:first-transverse-dirichlet-scale}
\end{equation}
The first positive lateral Neumann eigenvalue on
\(Q_{\ell_1,\ell_2}\) is
\begin{equation}
    \mu_\parallel(\ell_1,\ell_2)
    :=
    \min\left\{
        \frac{\pi^2}{\ell_1^2},
        \frac{\pi^2}{\ell_2^2}
    \right\}
    =
    \frac{\pi^2}{\max\{\ell_1,\ell_2\}^2}.
    \label{eq:first-lateral-neumann-scale}
\end{equation}
This is the first positive lateral increment above the lowest transverse
Dirichlet mode in \eqref{eq:mixed-cell-eigenvalues}.

\begin{definition}[Plate-scale saturation]
\label{def:plate-scale-saturation}
A rectangular cell \(R_{\ell_1,\ell_2,a}\) satisfying
\(\ell_1\ell_2=a^2\) is called plate-scale saturated if
\begin{equation}
    \mu_\parallel(\ell_1,\ell_2)=\lambda_\perp(a).
    \label{eq:plate-scale-saturation}
\end{equation}
\end{definition}

\begin{proposition}[Extremal lateral spectral scale]
\label{prop:spectral-selection-cube}
Among all rectangular cells
\[
    R_{\ell_1,\ell_2,a}
    =
    [0,\ell_1]\times[0,\ell_2]\times[0,a],
    \qquad
    \ell_1\ell_2=a^2,
\]
one has
\begin{equation}
    \mu_\parallel(\ell_1,\ell_2)
    \leq
    \lambda_\perp(a).
    \label{eq:lateral-gap-upper-bound}
\end{equation}
Moreover,
\begin{equation}
    \max_{\ell_1\ell_2=a^2}\mu_\parallel(\ell_1,\ell_2)
    =
    \lambda_\perp(a),
    \label{eq:max-lateral-gap}
\end{equation}
and the maximum is attained if and only if
\begin{equation}
    \ell_1=\ell_2=a.
    \label{eq:cube-equality-condition}
\end{equation}
Thus \(C_a=[0,a]^3\) is the unique plate-scale saturated cell in this
rectangular product class.
\end{proposition}

\begin{proof}
By \eqref{eq:first-lateral-neumann-scale},
\[
    \mu_\parallel(\ell_1,\ell_2)
    =
    \frac{\pi^2}{\max\{\ell_1,\ell_2\}^2}.
\]
Since \(\ell_1\ell_2=a^2\), one has
\[
    \max\{\ell_1,\ell_2\}\geq a.
\]
Therefore
\[
    \mu_\parallel(\ell_1,\ell_2)
    \leq
    \frac{\pi^2}{a^2}
    =
    \lambda_\perp(a).
\]
Equality holds if and only if
\(\max\{\ell_1,\ell_2\}=a\).  Under the constraint
\(\ell_1\ell_2=a^2\), this is equivalent to
\(\ell_1=\ell_2=a\).  Hence the maximum of
\(\mu_\parallel\) over the admissible rectangular class is
\(\lambda_\perp(a)\), and the maximizer is unique.
\end{proof}

\begin{corollary}[Aspect-ratio form]
\label{cor:aspect-ratio-gap}
For the parametrization
\[
    \ell_1=\alpha a,
    \qquad
    \ell_2=\alpha^{-1}a,
    \qquad
    \alpha>0,
\]
one has
\begin{equation}
    \frac{\mu_\parallel(\alpha a,\alpha^{-1}a)}
         {\lambda_\perp(a)}
    =
    \min\{\alpha^2,\alpha^{-2}\}.
    \label{eq:aspect-ratio-gap-ratio}
\end{equation}
Thus the lateral spectral scale is maximized exactly at
\(\alpha=1\).
\end{corollary}

\begin{proof}
Substituting
\(\ell_1=\alpha a\) and \(\ell_2=\alpha^{-1}a\) into
\eqref{eq:first-lateral-neumann-scale} gives
\[
    \mu_\parallel(\alpha a,\alpha^{-1}a)
    =
    \min\left\{
        \frac{\pi^2}{\alpha^2a^2},
        \frac{\pi^2\alpha^2}{a^2}
    \right\}.
\]
Dividing by \(\lambda_\perp(a)=\pi^2/a^2\) gives
\eqref{eq:aspect-ratio-gap-ratio}.  The function
\(\min\{\alpha^2,\alpha^{-2}\}\) is at most \(1\), with equality if and only
if \(\alpha=1\).
\end{proof}

\subsection{Mixed-cell heat trace and artificial-boundary scale}
\label{sec:mixed-cell-heat-trace}

The preceding selection uses the first positive lateral eigenvalue.  The
same cubical cell is also selected by the first shape-dependent coefficient
in the short-time heat trace of the full mixed cell operator
\(L_{\ell_1,\ell_2,a}^{\mathrm{cell}}\).  This connects the reference-cell
choice to the heat-trace structure used in the scalar Casimir finite-part
construction.

Define the one-dimensional Neumann and Dirichlet heat sums
\begin{equation}
    \Theta_N(\ell;t)
    :=
    \sum_{m=0}^{\infty}
    e^{-\pi^2m^2t/\ell^2},
    \qquad
    \Theta_D(\ell;t)
    :=
    \sum_{r=1}^{\infty}
    e^{-\pi^2r^2t/\ell^2}.
    \label{eq:theta-ND-definitions}
\end{equation}
By separation of variables,
\begin{equation}
    K_{\ell_1,\ell_2,a}(t)
    :=
    \Tr\!\left(e^{-tL_{\ell_1,\ell_2,a}^{\mathrm{cell}}}\right)
    =
    \Theta_N(\ell_1;t)\Theta_N(\ell_2;t)\Theta_D(a;t).
    \label{eq:mixed-cell-heat-trace-factorization}
\end{equation}

\begin{proposition}[Short-time heat trace and cubical minimization]
\label{prop:heat-trace-cube-selection}
As \(t\downarrow0\),
\begin{equation}
    K_{\ell_1,\ell_2,a}(t)
    =
    \frac{\ell_1\ell_2a}{8\pi^{3/2}}t^{-3/2}
    +
    \frac{a(\ell_1+\ell_2)-\ell_1\ell_2}{8\pi}t^{-1}
    +
    O_{\ell_1,\ell_2,a}(t^{-1/2}).
    \label{eq:mixed-cell-heat-trace-expansion}
\end{equation}
Under the plate-area constraint \(\ell_1\ell_2=a^2\), this becomes
\begin{equation}
    K_{\ell_1,\ell_2,a}(t)
    =
    \frac{a^3}{8\pi^{3/2}}t^{-3/2}
    +
    B(\ell_1,\ell_2,a)t^{-1}
    +
    O_{\ell_1,\ell_2,a}(t^{-1/2}),
    \label{eq:mixed-cell-heat-trace-expansion-area}
\end{equation}
where
\begin{equation}
    B(\ell_1,\ell_2,a)
    :=
    \frac{a(\ell_1+\ell_2)-a^2}{8\pi}.
    \label{eq:mixed-cell-B-coefficient}
\end{equation}
The coefficient \(B(\ell_1,\ell_2,a)\) is minimized over
\(\ell_1\ell_2=a^2\) if and only if
\[
    \ell_1=\ell_2=a.
\]
At the minimum,
\begin{equation}
    B(a,a,a)=\frac{a^2}{8\pi}.
    \label{eq:B-cube-value}
\end{equation}
Equivalently, the artificial lateral contribution
\[
    B_{\mathrm{lat}}(\ell_1,\ell_2,a)
    :=
    \frac{a(\ell_1+\ell_2)}{8\pi}
    \label{eq:lateral-heat-coefficient}
\]
is uniquely minimized by the cube.
\end{proposition}

\begin{proof}
The Jacobi transformation gives, as \(t\downarrow0\),
\begin{equation}
    \Theta_N(\ell;t)
    =
    \frac{\ell}{2\sqrt{\pi}}t^{-1/2}
    +
    \frac12
    +
    O_\ell(t^{-1/2}e^{-\ell^2/t}),
    \label{eq:theta-neumann-small-t}
\end{equation}
and
\begin{equation}
    \Theta_D(\ell;t)
    =
    \frac{\ell}{2\sqrt{\pi}}t^{-1/2}
    -
    \frac12
    +
    O_\ell(t^{-1/2}e^{-\ell^2/t}).
    \label{eq:theta-dirichlet-small-t}
\end{equation}
Substituting these expansions into
\eqref{eq:mixed-cell-heat-trace-factorization} gives the leading term
\[
    \frac{\ell_1}{2\sqrt{\pi}}t^{-1/2}
    \frac{\ell_2}{2\sqrt{\pi}}t^{-1/2}
    \frac{a}{2\sqrt{\pi}}t^{-1/2}
    =
    \frac{\ell_1\ell_2a}{8\pi^{3/2}}t^{-3/2}.
\]
The \(t^{-1}\) coefficient is the sum of the three terms in which two
factors contribute their leading \(t^{-1/2}\) term and one factor contributes
the constant boundary term:
\[
    \frac{a\ell_1}{8\pi}
    +
    \frac{a\ell_2}{8\pi}
    -
    \frac{\ell_1\ell_2}{8\pi}
    =
    \frac{a(\ell_1+\ell_2)-\ell_1\ell_2}{8\pi}.
\]
This proves \eqref{eq:mixed-cell-heat-trace-expansion}.

Under the area constraint \(\ell_1\ell_2=a^2\), the leading volume
coefficient is fixed, and
\[
    B(\ell_1,\ell_2,a)
    =
    \frac{a(\ell_1+\ell_2)-a^2}{8\pi}.
\]
By the arithmetic-geometric mean inequality,
\[
    \ell_1+\ell_2\geq 2\sqrt{\ell_1\ell_2}=2a,
\]
with equality if and only if \(\ell_1=\ell_2=a\).  Therefore
\[
    B(\ell_1,\ell_2,a)
    \geq
    \frac{2a^2-a^2}{8\pi}
    =
    \frac{a^2}{8\pi},
\]
with equality if and only if the cell is cubical.  The same AM--GM argument
applied to \(B_{\mathrm{lat}}\) proves the final statement.
\end{proof}

\begin{remark}[Boundary interpretation of the \(t^{-1}\) coefficient]
The \(t^{-1}\) coefficient in
\eqref{eq:mixed-cell-heat-trace-expansion} is the mixed-boundary surface
coefficient for the rectangular cell.  The term
\(a(\ell_1+\ell_2)/(8\pi)\) comes from the four artificial lateral Neumann
faces, whose total area is \(2a(\ell_1+\ell_2)\).  The term
\(-\ell_1\ell_2/(8\pi)\) comes from the two physical Dirichlet plate faces,
whose total area is \(2\ell_1\ell_2\).  Under
\(\ell_1\ell_2=a^2\), the physical Dirichlet contribution is fixed, while
the artificial lateral contribution is uniquely minimized by the cube.
\end{remark}

\subsection{Flat Green energy and aspect-ratio monotonicity}
\label{sec:flat-green-aspect-monotonicity}

Let \(D\subset\R^3\) be a bounded measurable reference cell with
\(|D|=1\).  Define
\begin{equation}
    \Delta_D
    :=
    \int_D\int_D\frac{\dd^3x\,\dd^3y}{|x-y|}.
    \label{eq:reference-cell-delta}
\end{equation}
This integral is finite.  If \(D\) has diameter bounded by \(R\), then for
each fixed \(x\in D\),
\[
    \int_D\frac{\dd^3y}{|x-y|}
    \leq
    \int_{|z|\leq R}\frac{\dd^3z}{|z|}
    =
    2\pi R^2.
\]
Hence \(\Delta_D<\infty\).

Let \(L_0=-\Delta\) on \(\R^3\).  With the positive Laplacian convention,
\[
    L_0^{-1}(x,y)=\frac{1}{4\pi|x-y|}.
\]
Therefore
\begin{equation}
    \Delta_D
    =
    4\pi
    \braket{\chi_D}{L_0^{-1}\chi_D}_{L^2(\R^3)}.
    \label{eq:reference-cell-green}
\end{equation}
Thus \(\Delta_D\) is the deterministic flat Green energy of the
unit-density reference source \(\chi_D\), written using the inverse-distance
kernel rather than the operator-normalized kernel
\((4\pi)^{-1}|x-y|^{-1}\).

For the plate-compatible aspect-ratio family \eqref{eq:unit-reference-aspect-cell}, write
\begin{equation}
    \Delta(\alpha):=\Delta_{D_\alpha},
    \qquad
    D_\alpha=[0,\alpha]\times[0,\alpha^{-1}]\times[0,1].
    \label{eq:delta-alpha-definition}
\end{equation}
The following proposition shows that the cubical cell is also selected by the
flat Green-energy functional itself within this plate-compatible rectangular
family.

\begin{proposition}[Aspect-ratio monotonicity of the flat Green energy]
\label{prop:aspect-ratio-green-monotonicity}
For \(\alpha>0\), let \(\Delta(\alpha)\) be defined by
\eqref{eq:delta-alpha-definition}.  Then
\begin{equation}
    \Delta(\alpha)=\Delta(\alpha^{-1}).
    \label{eq:delta-alpha-symmetry}
\end{equation}
Moreover, the function
\[
    \beta\longmapsto \Delta(e^\beta)
\]
is strictly decreasing for \(\beta>0\).  Consequently,
\begin{equation}
    \Delta(\alpha)\leq \Delta(1),
    \label{eq:delta-alpha-cube-max}
\end{equation}
with equality if and only if \(\alpha=1\).
\end{proposition}

\begin{proof}
The symmetry \(\Delta(\alpha)=\Delta(\alpha^{-1})\) follows by interchanging
the first two coordinate axes.

For \(L>0\) and \(t>0\), define the one-dimensional Gaussian interval
overlap
\begin{equation}
    I_L(t):=\int_0^L\int_0^L e^{-t(x-y)^2}\,\dd x\,\dd y.
    \label{eq:interval-overlap-main}
\end{equation}
Using
\[
    \frac{1}{|x-y|}
    =
    \frac{1}{\sqrt\pi}
    \int_0^\infty t^{-1/2}e^{-t|x-y|^2}\,\dd t,
\]
and Tonelli's theorem for the nonnegative integrand, we obtain
\begin{equation}
    \Delta(\alpha)
    =
    \frac{1}{\sqrt\pi}
    \int_0^\infty
    t^{-1/2}
    I_\alpha(t)I_{\alpha^{-1}}(t)I_1(t)\,\dd t.
    \label{eq:delta-alpha-overlap-representation}
\end{equation}
By Lemma~\ref{lem:gaussian-interval-log-concavity} in
Appendix~\ref{app:interval-overlap-log-concavity}, for every fixed \(t>0\)
the function
\[
    H_t(u):=\log I_{e^u}(t)
\]
is strictly concave in \(u\).  Hence, for \(\beta>0\),
\[
    \frac{\dd}{\dd\beta}
    \left(H_t(\beta)+H_t(-\beta)\right)
    =
    H_t'(\beta)-H_t'(-\beta)<0,
\]
because \(H_t'\) is strictly decreasing.  Therefore
\[
    I_{e^\beta}(t)I_{e^{-\beta}}(t)
\]
is strictly decreasing in \(\beta>0\) for every \(t>0\).  Since
\(I_1(t)>0\), the integrand in
\eqref{eq:delta-alpha-overlap-representation} is pointwise strictly
decreasing as a function of \(\beta>0\).  Integrating against the positive
measure \(\pi^{-1/2}t^{-1/2}\dd t\) gives that
\(\Delta(e^\beta)\) is strictly decreasing for \(\beta>0\).  The maximum is
therefore attained uniquely at \(\beta=0\), equivalently \(\alpha=1\).
\end{proof}

For the unit cube
\[
    C=[0,1]^3,
\]
we write
\begin{equation}
    \Delta_3(-1):=\Delta_C=\Delta(1)
    =
    \int_C\int_C\frac{\dd^3x\,\dd^3y}{|x-y|}.
    \label{eq:delta3-definition}
\end{equation}
Equivalently,
\begin{equation}
    \Delta_3(-1)
    =
    4\pi
    \braket{\chi_C}{L_0^{-1}\chi_C}_{L^2(\R^3)}.
    \label{eq:delta3-green}
\end{equation}
The known closed form is
\begin{align}
\Delta_3(-1)
&=
\frac{2}{5}(1+\sqrt2-2\sqrt3)-\frac{2\pi}{3}
-6\log2+2\log(1+\sqrt2)
\notag\\
&\quad
+12\log(1+\sqrt3)-4\log(2+\sqrt3).
\label{eq:delta3-closed-form}
\end{align}

\subsection{Scaling and deterministic reference energy}

For \(a>0\), set
\[
    D_a:=aD.
\]
The inverse-distance integral scales as
\begin{equation}
    \int_{D_a}\int_{D_a}\frac{\dd^3x\,\dd^3y}{|x-y|}
    =
    a^5\Delta_D,
    \label{eq:reference-delta-scaling}
\end{equation}
by the change of variables \(x=au\), \(y=av\).

Let two scalar test densities be uniformly distributed over \(D_a\) with
opposite total weights \(Q\) and \(-Q\):
\[
    \rho_{+,D,a}(x)=\frac{Q}{a^3}\chi_{D_a}(x),
    \qquad
    \rho_{-,D,a}(x)=-\frac{Q}{a^3}\chi_{D_a}(x),
\]
where \(|D|=1\).  For the reference inverse-distance pair kernel
\[
    K_\gamma(x,y)=\frac{\gamma}{|x-y|},
    \qquad
    \gamma>0,
\]
the mutual bilinear Green-kernel functional of these two reference densities
is
\begin{equation}
\begin{split}
    U_D(\gamma,Q,a)
    &:=
    \int_{D_a}\int_{D_a}
    \rho_{+,D,a}(x)K_\gamma(x,y)\rho_{-,D,a}(y)
    \dd^3x\,\dd^3y                                      \\
    &=
    -\frac{\gamma Q^2}{a}\Delta_D.
\end{split}
    \label{eq:reference-cell-energy}
\end{equation}
Set
\[
    \mathcal Q:=\gamma Q^2.
\]
Then
\begin{equation}
    U_D(\mathcal Q,a)
    =
    -\frac{\mathcal Q}{a}\Delta_D.
    \label{eq:reference-cell-energy-Qcal}
\end{equation}

For \(n^2\) identical cells, define the additive deterministic reference
energy
\begin{equation}
    U_D(\mathcal Q,n,a)
    :=
    n^2U_D(\mathcal Q,a)
    =
    -\frac{n^2}{a}\mathcal Q\Delta_D.
    \label{eq:reference-n2-cell-energy}
\end{equation}
For the aspect-ratio family, write
\begin{equation}
    U_\alpha(\mathcal Q,n,a)
    :=
    U_{D_\alpha}(\mathcal Q,n,a)
    =
    -\frac{n^2}{a}\mathcal Q\Delta(\alpha).
    \label{eq:aspect-reference-energy}
\end{equation}
By Proposition~\ref{prop:aspect-ratio-green-monotonicity}, for fixed
positive \(\mathcal Q\), \(n\), and \(a\), the magnitude
\(|U_\alpha(\mathcal Q,n,a)|\) is maximized in the plate-compatible
rectangular aspect-ratio family at \(\alpha=1\), uniquely.

For the selected cubical cell \(C\), define
\begin{equation}
    U_{\Delta}(\mathcal Q,n,a)
    :=
    U_C(\mathcal Q,n,a)
    =
    -\frac{n^2}{a}\mathcal Q\Delta_3(-1).
    \label{eq:n2-cube-energy}
\end{equation}
The scaling in \eqref{eq:n2-cube-energy} matches the scaling of the scalar
plate finite part in \eqref{eq:reference-plate-energy-na}: both are of the
form constant times \(n^2/a\).

\subsection{Reference-cell scalar comparison coefficient}
\label{sec:trace-ratio}

For a plate configuration with \(A=n^2a^2\), choose the reference
normalization
\[
    \mathcal Q=\hbar c.
\]
For the aspect-ratio family \(D_\alpha\), define
\begin{equation}
    \Theta_{\alpha,\tau}^{(N)}
    :=
    \frac{U_\tau^{(N)}}{U_\alpha(\hbar c,n,a)}
    =
    -\frac{a}{n^2\hbar c\,\Delta(\alpha)}U_\tau^{(N)}.
    \label{eq:regulated-theta-alpha-ratio}
\end{equation}
For the selected cubical reference cell, write
\begin{equation}
    \Theta_{\Delta,\tau}^{(N)}
    :=
    \Theta_{1,\tau}^{(N)}
    =
    \frac{U_\tau^{(N)}}{U_{\Delta}(\hbar c,n,a)}
    =
    -\frac{a}{n^2\hbar c\,\Delta_3(-1)}U_\tau^{(N)}.
    \label{eq:regulated-theta-ratio}
\end{equation}
The raw regulated quantity \(U_\tau^{(N)}\) is divergent as
\(\tau\to0^+\), so the comparison coefficient is defined at the level of
finite-part expectation:
\begin{equation}
    \overline{\Theta}_{\alpha,\mathcal R}^{(N)}
    :=
    \FP_{\mathcal R,\,\tau\to0^+}
    \E[\Theta_{\alpha,\tau}^{(N)}].
    \label{eq:expected-theta-alpha-def}
\end{equation}
For \(\alpha=1\), set
\begin{equation}
    \overline{\Theta}_{\Delta,\mathcal R}^{(N)}
    :=
    \overline{\Theta}_{1,\mathcal R}^{(N)}
    =
    \FP_{\mathcal R,\,\tau\to0^+}
    \E[\Theta_{\Delta,\tau}^{(N)}].
    \label{eq:expected-theta-def}
\end{equation}
These definitions concern finite-part expectations.  They do not require a
separately constructed renormalized random variable obtained from
\(\Theta_{\alpha,\tau}^{(N)}\) or \(\Theta_{\Delta,\tau}^{(N)}\)
configuration by configuration.

\begin{proposition}[Aspect-ratio comparison coefficient and cubical minimum]
\label{prop:expected-theta}
For \(N\) independent scalar channels in the large-area parallel-plate limit,
using the standard scalar finite-part subtraction and the normalization
\(\mathcal Q=\hbar c\), the plate-compatible rectangular aspect-ratio family
satisfies
\begin{equation}
    \overline{\Theta}_{\alpha}^{(N)}
    =
    \frac{N\pi^2}{1440\Delta(\alpha)}.
    \label{eq:expected-theta-alpha-N}
\end{equation}
Consequently,
\begin{equation}
    \overline{\Theta}_{\alpha}^{(N)}
    \geq
    \overline{\Theta}_{\Delta}^{(N)}
    =
    \frac{N\pi^2}{1440\Delta_3(-1)},
    \label{eq:expected-theta-N}
\end{equation}
with equality if and only if \(\alpha=1\).  In particular, for two scalar
channels,
\begin{equation}
    \overline{\Theta}_{\Delta}^{(2)}
    =
    \frac{\pi^2}{720\Delta_3(-1)}.
    \label{eq:expected-theta-two-channel}
\end{equation}
\end{proposition}

\begin{proof}
By \eqref{eq:regulated-theta-alpha-ratio},
\[
    \E[\Theta_{\alpha,\tau}^{(N)}]
    =
    -\frac{a}{n^2\hbar c\,\Delta(\alpha)}
    \E[U_\tau^{(N)}].
\]
Taking the same finite part used in the scalar parallel-plate computation
and using \eqref{eq:reference-plate-energy-na} gives
\[
    \overline{\Theta}_{\alpha}^{(N)}
    =
    -\frac{a}{n^2\hbar c\,\Delta(\alpha)}
    \left(
        -\frac{n^2}{a}\frac{N\pi^2\hbar c}{1440}
    \right)
    =
    \frac{N\pi^2}{1440\Delta(\alpha)}.
\]
This proves \eqref{eq:expected-theta-alpha-N}.  By
Proposition~\ref{prop:aspect-ratio-green-monotonicity},
\(\Delta(\alpha)\leq\Delta(1)=\Delta_3(-1)\), with equality if and only if
\(\alpha=1\).  Taking reciprocals gives the minimum statement
\eqref{eq:expected-theta-N}.  The two-channel formula follows by setting
\(N=2\).
\end{proof}

\begin{remark}[What is being compared]
The numerator in \eqref{eq:regulated-theta-alpha-ratio} is the stochastic
quadratic form whose finite-part expectation equals the scalar plate finite
part.  The denominator is the deterministic flat Green energy of \(n^2\)
rectangular reference cells scaled by the same separation length \(a\), with
normalization \(\mathcal Q=\hbar c\).  Within the plate-compatible rectangular
aspect-ratio family, the cubical cell is not inserted as an arbitrary shape:
it saturates the transverse Dirichlet scale, minimizes the leading
artificial-boundary coefficient in the mixed cell heat trace, and maximizes
the deterministic flat Green energy \(\Delta(\alpha)\).  Equivalently, it
minimizes the associated finite-part comparison coefficient.
\end{remark}

\begin{remark}[Common source of the three selection criteria]
The three criteria used above are not independent assumptions imposed on the
reference cell.  They are three consequences of the same constrained
aspect-ratio problem.  The admissible rectangular family
\[
    D_\alpha=[0,\alpha]\times[0,\alpha^{-1}]\times[0,1]
\]
is invariant under the involution \(\alpha\mapsto\alpha^{-1}\), and the
cubical cell \(D_1=C\) is the fixed point of this symmetry.  The lateral
spectral-gap and mixed heat-trace criteria reduce to arithmetic-geometric
mean inequalities under the constraint \(\ell_1\ell_2=a^2\), while the flat
Green-energy criterion follows from the strict log-concavity of Gaussian
interval overlaps proved in Appendix~\ref{app:interval-overlap-log-concavity}.
Thus the cube is selected in three compatible ways within the same
plate-adapted rectangular class: it maximizes the free lateral spectral
scale, minimizes the first shape-dependent artificial-boundary heat-trace
coefficient, and maximizes the deterministic flat Green energy
\(\Delta(\alpha)\).  Equivalently, it minimizes the associated finite-part
comparison coefficient.
\end{remark}

\begin{remark}[Conditional nature of the cell selection]
The selection of \(C_a=[0,a]^3\) is conditional on the rectangular product
cell class \eqref{eq:reference-rectangular-cell}, the plate-area constraint
\eqref{eq:reference-area-constraint}, and the free-boundary quadratic form
on artificial lateral faces.  The maximization of \(\Delta(\alpha)\) and the
minimization of \(\overline{\Theta}_{\alpha}^{(N)}\) are asserted only for
the plate-compatible rectangular aspect-ratio family
\(D_\alpha=[0,\alpha]\times[0,\alpha^{-1}]\times[0,1]\).  No assertion is
made that the cube optimizes \(\Delta_D\) over arbitrary unit-volume subsets
of \(\R^3\).
\end{remark}

\begin{remark}[Inverse of a finite-part expectation]
The inverse associated with \eqref{eq:expected-theta-two-channel} is the
inverse of the finite-part expectation,
\((\overline{\Theta}_{\Delta}^{(2)})^{-1}\).  It is not a statement about
\(\E[\Theta^{-1}]\) for a random variable.  In general,
\(\E[X^{-1}]\neq(\E[X])^{-1}\).
\end{remark}

\section{Normalization and scalar-channel conventions}
\label{sec:model-conventions}

This section records the conventions that connect the preceding formulas.  They are included to keep separate the numerical equalities proved above from additional physical identifications that would require further structure.

\subsection{Reference inverse-distance Green-kernel convention}

The reference-cell functional uses the kernel \(|x-y|^{-1}\) because this is
the Green-kernel scaling of \(L_0^{-1}\) in three flat dimensions.  With
\(L_0=-\Delta\) on \(\R^3\),
\[
    L_0^{-1}(x,y)=\frac{1}{4\pi|x-y|}.
\]
Thus \(\Delta_D=4\pi\braket{\chi_D}{L_0^{-1}\chi_D}\) is a deterministic
flat Green energy written in inverse-distance normalization.  Writing a
reference pair kernel as \(\gamma |x-y|^{-1}\) fixes the normalization of
these deterministic reference pairings.  The operator identity proved in
Theorem~\ref{thm:brane-restriction} is instead the statement
\[
    V_{B,5/2}=gL_B^{-1},
\]
with \(g=\kappa/(6\pi^2)\).  Choosing numerical values for \(g\), \(\gamma\),
or \(\mathcal Q=\gamma Q^2\) is additional input beyond the operator
restriction theorem.

\subsection{Scalar source convention}

The stochastic identity uses the generalized scalar source with covariance form
\[
    \E[\sigma(\varphi)\overline{\sigma(\psi)}]
    =\frac{\hbar c}{g}\braket{L_B^{3/4}\varphi}{L_B^{3/4}\psi}.
\]
This covariance is chosen so that the quadratic form with $gL_B^{-1}$ produces the scalar trace with one power $L_B^{1/2}$.  The equality
\[
    \E[U_\tau]
    =\frac{\hbar c}{2}\Tr(L_B^{1/2}e^{-\tau L_B})
\]
is therefore a trace identity under the stated covariance convention.

\subsection{Scalar-channel convention}

The integer $N$ counts independent scalar channels in the trace identity.  Thus
\[
    \frac{1}{A}\E[U_{\mathcal R}^{(N)}]
    =-\frac{N\pi^2\hbar c}{1440a^3}
\]
in the large-area parallel-plate limit with the standard scalar finite-part subtraction.  The case $N=2$ is a scalar-channel doubling of the one-channel coefficient.  A vector-field treatment with transverse modes and boundary conditions is a different operator problem.

\subsection{Reference-cell comparison convention}

The coefficient \(\overline{\Theta}_{\Delta}^{(N)}\) is formed only after
three choices have been specified: the scalar finite-part prescription, the
reference normalization \(\mathcal Q=\hbar c\), and the plate-compatible
reference-cell class.  In the rectangular product family
\[
    D_\alpha=[0,\alpha]\times[0,\alpha^{-1}]\times[0,1],
    \qquad
    \alpha>0,
\]
with free artificial lateral faces and Dirichlet plate faces, the cubical
cell \(D_1=C\) is selected by
Propositions~\ref{prop:spectral-selection-cube},
\ref{prop:heat-trace-cube-selection}, and
\ref{prop:aspect-ratio-green-monotonicity}.  The comparison coefficient may
first be evaluated on the aspect-ratio family:
\[
    \overline{\Theta}_{\alpha}^{(N)}
    =
    \frac{N\pi^2}{1440\Delta(\alpha)}.
\]
Since \(\Delta(\alpha)\) is uniquely maximized at \(\alpha=1\), this
coefficient is uniquely minimized at the cubical cell.  The displayed
cubical value is therefore
\[
    \overline{\Theta}_{\Delta}^{(N)}
    =
    \overline{\Theta}_{1}^{(N)}
    =
    \frac{N\pi^2}{1440\Delta_3(-1)}.
\]
This coefficient records a Green-energy calibration of the scalar plate
finite part relative to the selected cubical reference cell.  It is not an
additional finite-part prescription and it does not alter the stochastic
trace identity.

\section{Scope and limitations}
\label{sec:scope-limitations}

The results above are scalar spectral identities.  They do not constitute a derivation of the electromagnetic Casimir effect, which requires the appropriate vector operator, gauge constraints, and boundary conditions.  They also do not determine the normalization constants $g$, $\kappa$, $\gamma$, or $\mathcal Q$ from first principles.  The fractional ambient operator $L_M^{-5/2}$ is a Riesz-type model operator, not the ordinary Green operator of a local six-dimensional Laplacian.  Finally, the reference-cell comparison coefficient is conditional on the plate-compatible rectangular product class, the free-boundary treatment of artificial lateral faces, and the normalization $\mathcal Q=\hbar c$.  The cubical value is an extremal value only within the plate-compatible rectangular aspect-ratio family $D_\alpha=[0,\alpha]\times[0,\alpha^{-1}]\times[0,1]$; no optimality over arbitrary unit-volume subsets of $\R^3$ is asserted.

\section{Summary and conclusion}
\label{sec:summary}

The construction can be compressed into four formal packages.

\paragraph{Operator/restriction package.}
Take
\[
    M=B\times\R^3_\perp,
    \qquad
    L_M=L_B-\Delta_\perp,
    \qquad
    V_M=\kappa L_M^{-5/2}.
\]
Define the brane restriction by the operator-valued transverse momentum integral
\[
    V_{B,5/2}
    =\kappa\int_{\R^3}\frac{\dd^3q}{(2\pi)^3}(L_B+|q|^2)^{-5/2}.
\]
Then
\[
    V_{B,5/2}=gL_B^{-1},
    \qquad
    g=\frac{\kappa}{6\pi^2}.
\]
More generally, in codimension $m$, the restriction of $(L_B-\Delta_\perp)^{-s}$ is proportional to $L_B^{m/2-s}$ whenever $s>m/2$.

\paragraph{Source/statistics package.}
Take
\[
    \sigma_\tau=\left(\frac{\hbar c}{g}\right)^{1/2}L_B^{3/4}e^{-\tau L_B/2}\xi.
\]
Then
\[
    \E[\sigma_\tau\otimes\sigma_\tau^*]
    =\frac{\hbar c}{g}L_B^{3/2}e^{-\tau L_B}.
\]
The exponent $3/4$ is fixed by the requirement that this covariance contribute $L_B^{3/2}$ when paired with the brane Green operator $gL_B^{-1}$.

\paragraph{Renormalization package.}
Use the same heat-kernel regulator and the same specified finite-part subtraction for the stochastic interaction and the scalar spectral trace.  Then
\[
    \E[U_\tau]
    =\frac{\hbar c}{2}\Tr(L_B^{1/2}e^{-\tau L_B})
\]
for $\tau>0$, and hence
\[
    \E[U_{\mathcal{R}}]=E_{\mathrm{Cas},\mathcal{R}}(L_B).
\]
For parallel scalar plates,
\[
    \frac{1}{A}\E[U_{\mathcal{R}}^{(1)}]
    =-\frac{\pi^2\hbar c}{1440a^3}.
\]
For two scalar channels,
\[
    \frac{1}{A}\E[U_{\mathcal{R}}^{(2)}]
    =-\frac{\pi^2\hbar c}{720a^3}.
\]

\paragraph{Reference Green-energy package.}
Write the plate area as \(A=n^2a^2\).  In the plate-compatible rectangular
cell class
\[
    R_{\ell_1,\ell_2,a}
    =
    [0,\ell_1]\times[0,\ell_2]\times[0,a],
    \qquad
    \ell_1\ell_2=a^2,
\]
the artificial lateral faces are treated by the unconstrained \(H^1\)
energy form, hence by the Neumann realization.  The first positive lateral
Neumann scale satisfies
\[
    \mu_\parallel(\ell_1,\ell_2)
    \leq
    \frac{\pi^2}{a^2},
\]
with equality if and only if \(\ell_1=\ell_2=a\).  The same cubical cell
also minimizes the \(t^{-1}\) artificial-boundary coefficient in the mixed
cell heat trace.

For the same aspect-ratio family
\[
    D_\alpha=[0,\alpha]\times[0,\alpha^{-1}]\times[0,1],
\]
the deterministic flat Green energy
\[
    \Delta(\alpha)
    =
    \int_{D_\alpha}\int_{D_\alpha}\frac{\dd^3x\,\dd^3y}{|x-y|}
\]
is uniquely maximized at \(\alpha=1\).  Therefore the associated comparison
coefficient is uniquely minimized at the cubical cell.  Thus the selected
reference cell is \(C_a=[0,a]^3\).

For \(C=[0,1]^3\), define
\[
    \Delta_3(-1)
    =
    \int_C\int_C\frac{\dd^3x\,\dd^3y}{|x-y|}
    =
    4\pi\braket{\chi_C}{L_0^{-1}\chi_C}.
\]
For \(n^2\) cubical reference cells and normalization \(\mathcal Q=\hbar c\),
the deterministic reference energy is
\[
    U_{\Delta}(\hbar c,n,a)
    =
    -\frac{n^2}{a}\hbar c\,\Delta_3(-1).
\]
Comparing the scalar finite-part stochastic plate energy with this selected
reference Green energy gives the cubical minimum of the aspect-ratio family,
\[
    \overline{\Theta}_{\Delta}^{(N)}
    =
    \min_{\alpha>0}\overline{\Theta}_{\alpha}^{(N)}
    =
    \frac{N\pi^2}{1440\Delta_3(-1)}.
\]

The construction is therefore a chain of scalar spectral identities: a codimension-three restricted Riesz mediator gives the brane Green operator, and the heat-regularized Gaussian scalar source turns the corresponding quadratic form into the scalar finite-part trace in expectation.

\appendix

\section{Derivation of the scalar parallel-plate coefficient}
\label{app:parallel-zeta}

This appendix recalls the standard zeta-regularized calculation for one scalar Dirichlet channel between parallel plates.  Let the plates be separated by distance $a$, and take the large-area limit in the two lateral directions.  Formally,
\begin{equation}
    \frac{E(s)}{A}
    =\frac{\hbar c}{2}\mu^{2s}\sum_{n=1}^{\infty}\int_{\R^2}\frac{\dd^2k}{(2\pi)^2}
    \left(k^2+\left(\frac{\pi n}{a}\right)^2\right)^{1/2-s},
    \label{eq:zeta-regularized-plate}
\end{equation}
where $s$ is initially taken large enough for convergence and then analytically continued to $s=0$.  The parameter $\mu$ keeps dimensions fixed and drops out of the finite plate-dependent term at $s=0$.

Using the dimensional integral identity
\begin{equation}
    \int_{\R^d}\frac{\dd^dk}{(2\pi)^d}(k^2+m^2)^{-\nu}
    =\frac{1}{(4\pi)^{d/2}}\frac{\Gamma(\nu-d/2)}{\Gamma(\nu)}(m^2)^{d/2-\nu},
    \label{eq:dimensional-integral}
\end{equation}
with $d=2$ and $\nu=s-1/2$, we obtain
\begin{align}
    \frac{E(s)}{A}
    &=\frac{\hbar c}{2}\mu^{2s}\frac{1}{4\pi}
      \frac{\Gamma(s-3/2)}{\Gamma(s-1/2)}
      \sum_{n=1}^{\infty}
      \left(\frac{\pi n}{a}\right)^{3-2s}
      \notag\\
    &=\frac{\hbar c}{8\pi}\mu^{2s}
      \frac{\Gamma(s-3/2)}{\Gamma(s-1/2)}
      \left(\frac{\pi}{a}\right)^{3-2s}\zeta(2s-3).
    \label{eq:plate-zeta-expression}
\end{align}
At $s=0$,
\[
    \frac{\Gamma(-3/2)}{\Gamma(-1/2)}=-\frac{2}{3},
    \qquad
    \zeta(-3)=\frac{1}{120}.
\]
Therefore
\begin{equation}
    \frac{E(0)}{A}
    =\frac{\hbar c}{8\pi}\left(-\frac{2}{3}\right)
     \left(\frac{\pi}{a}\right)^3\frac{1}{120}
    =-\frac{\pi^2\hbar c}{1440a^3}.
    \label{eq:scalar-plate-result-app}
\end{equation}
This is \eqref{eq:scalar-casimir-per-area}.

\section{The cube integral and its closed form}
\label{app:box-integral}

The cube integral
\[
    \Delta_3(-1)=\int_{[0,1]^3}\int_{[0,1]^3}\frac{\dd^3x\,\dd^3y}{|x-y|}
\]
can be reduced to a three-dimensional integral by the change of variables $r=x-y$.  The difference $r_i$ in each coordinate ranges over $[-1,1]$, and the measure of pairs with coordinate difference $r_i$ contributes a factor $(1-|r_i|)$.  Hence
\begin{equation}
    \Delta_3(-1)
    =\int_{[-1,1]^3}\frac{(1-|r_1|)(1-|r_2|)(1-|r_3|)}{(r_1^2+r_2^2+r_3^2)^{1/2}}\dd^3r.
    \label{eq:delta-reduced}
\end{equation}
The singularity at $r=0$ is integrable in three dimensions.  The known closed form is
\begin{align}
\Delta_3(-1)
&=\frac{2}{5}(1+\sqrt2-2\sqrt3)-\frac{2\pi}{3}-6\log2+2\log(1+\sqrt2)
\notag\\
&\quad+12\log(1+\sqrt3)-4\log(2+\sqrt3).
\end{align}
This is the same value quoted in \eqref{eq:delta3-closed-form}.  The integral is sometimes discussed in the literature on box integrals and line picking.

\section{Log-concavity of Gaussian interval overlaps}
\label{app:interval-overlap-log-concavity}

This appendix proves the one-dimensional log-concavity lemma used in
Proposition~\ref{prop:aspect-ratio-green-monotonicity}.  For \(L>0\) and
\(t>0\), set
\[
    I_L(t)=\int_0^L\int_0^L e^{-t(x-y)^2}\,\dd x\,\dd y.
\]

\begin{lemma}[Strict log-concavity in logarithmic length]
\label{lem:gaussian-interval-log-concavity}
For every fixed \(t>0\), the function
\[
    u\longmapsto \log I_{e^u}(t)
\]
is strictly concave on \(\R\).
\end{lemma}

\begin{proof}
By the change of variables \(X=\sqrt t\,x\), \(Y=\sqrt t\,y\),
\[
    I_L(t)=t^{-1}J(L\sqrt t),
\]
where
\[
    J(r):=\int_0^r\int_0^r e^{-(X-Y)^2}\,\dd X\,\dd Y
    =2\int_0^r(r-s)e^{-s^2}\,\dd s.
\]
Multiplying by the positive constant \(t^{-1}\) and translating the variable
\(u\) by \((1/2)\log t\) do not affect strict concavity.  It is therefore
enough to prove that
\[
    w\longmapsto \log J(e^w)
\]
is strictly concave.

Since the factor \(2\) in \(J\) is irrelevant for logarithmic concavity, set
\[
    j(r):=\frac{1}{2}J(r)=\int_0^r(r-s)e^{-s^2}\,\dd s.
\]
Introduce
\[
    A(r):=\int_0^r e^{-s^2}\,\dd s,
    \qquad
    E(r):=e^{-r^2},
    \qquad
    B(r):=\int_0^r s e^{-s^2}\,\dd s=\frac{1-E(r)}{2}.
\]
Then
\[
    j(r)=rA(r)-B(r),
    \qquad
    j'(r)=A(r),
    \qquad
    j''(r)=E(r).
\]
Let \(r=e^w\).  A direct differentiation gives
\begin{equation}
    \frac{\dd^2}{\dd w^2}\log j(e^w)
    =
    \frac{r(A(r)+rE(r))j(r)-r^2A(r)^2}{j(r)^2}.
    \label{eq:appendix-H-second-derivative}
\end{equation}
Thus it remains to prove that the numerator in
\eqref{eq:appendix-H-second-derivative} is negative for every \(r>0\).
Equivalently, since \(j(r)=rA(r)-B(r)\), we must prove
\begin{equation}
    B(r)(A(r)+rE(r))-r^2A(r)E(r)>0.
    \label{eq:appendix-core-positivity}
\end{equation}
Define
\[
    h(r):=(1-E(r))(A(r)+rE(r))-2r^2A(r)E(r).
\]
Then \eqref{eq:appendix-core-positivity} is equivalent to \(h(r)>0\).  Since
\(h(0)=0\), it is enough to show \(h'(r)>0\) for \(r>0\).  Differentiating,
using \(A'(r)=E(r)\) and \(E'(r)=-2rE(r)\), gives
\begin{equation}
    h'(r)=2E(r)k(r),
    \label{eq:appendix-h-prime}
\end{equation}
where
\begin{equation}
    k(r):=rA(r)(2r^2-1)+(1-r^2)(1-E(r)).
    \label{eq:appendix-k-definition}
\end{equation}
We prove \(k(r)>0\) in two ranges.

First let \(0<r\leq 1/\sqrt2\).  Then \(2r^2-1\leq0\), and
\(A(r)\leq r\), so
\[
    rA(r)(2r^2-1)
    \geq
    r^2(2r^2-1).
\]
Also, with \(x=r^2\), the elementary inequality
\(1-e^{-x}\geq x-x^2/2\) gives
\[
    1-E(r)\geq r^2-\frac{r^4}{2}.
\]
Therefore
\[
\begin{split}
    k(r)
    &\geq
    r^2(2r^2-1)
    +(1-r^2)\left(r^2-\frac{r^4}{2}\right)  \\
    &=
    \frac12 r^4(1+r^2)>0.
\end{split}
\]

Now let \(r\geq1/\sqrt2\).  Then \(2r^2-1\geq0\).  Since
\[
    1-E(r)
    =\int_0^r 2s e^{-s^2}\,\dd s
    \leq
    2r\int_0^r e^{-s^2}\,\dd s
    =2rA(r),
\]
we have \(A(r)\geq(1-E(r))/(2r)\).  Hence
\[
\begin{split}
    k(r)
    &\geq
    \frac{1-E(r)}{2}(2r^2-1)
    +(1-r^2)(1-E(r))  \\
    &=
    \frac{1-E(r)}{2}>0.
\end{split}
\]
Thus \(k(r)>0\) for all \(r>0\).  By \eqref{eq:appendix-h-prime},
\(h'(r)>0\) for all \(r>0\), and since \(h(0)=0\), \(h(r)>0\) for all
\(r>0\).  This proves \eqref{eq:appendix-core-positivity}, hence the second
derivative in \eqref{eq:appendix-H-second-derivative} is strictly negative.
Therefore \(w\mapsto\log J(e^w)\), and hence
\(u\mapsto\log I_{e^u}(t)\), is strictly concave.
\end{proof}

\section{Operator convention for the inverse-distance kernel normalization}
\label{app:inverse-distance-normalization}

On $\R^3$ with the positive Laplacian $L_0=-\Delta$, the Green operator satisfies
\[
    L_0^{-1}(x,y)=\frac{1}{4\pi|x-y|}.
\]
Thus a scalar pair kernel with operator normalization $\lambda L_0^{-1}$ has kernel
\[
    \frac{\lambda}{4\pi|x-y|}.
\]
For two scalar source densities $\rho_1,\rho_2$, the corresponding bilinear interaction is
\[
    \braket{\rho_1}{\lambda L_0^{-1}\rho_2}
    =\frac{\lambda}{4\pi}\int\int\frac{\rho_1(x)\rho_2(y)}{|x-y|}\dd^3x\dd^3y.
\]
If $\rho_+$ and $\rho_-$ are uniform opposite source densities of total weights $Q$ and $-Q$ over a cube of side $a$, this convention agrees with \eqref{eq:reference-cell-energy-Qcal} after setting
\[
    \gamma=\frac{\lambda}{4\pi},
    \qquad
    \mathcal Q=\frac{\lambda Q^2}{4\pi}.
\]
The normalization used in Section~\ref{sec:reference-green-normalization} is therefore the inverse-distance Green-kernel normalization associated with the scalar operator $L_0^{-1}$; for the cube reference cell it specializes to \eqref{eq:n2-cube-energy}.

\end{document}